\documentclass[11pt,a4paper]{article}
\usepackage{amsmath, amsfonts,amsthm,amssymb,amsbsy,upref,color,graphicx,amscd,hyperref,makeidx,enumerate}
\usepackage[active]{srcltx}
\usepackage{bbm} 
\usepackage[latin1]{inputenc}

\newcommand{\ra}{\rightarrow}

\newcommand{\g}{\gamma}
\newcommand{\Om}{\Omega}

\newcommand{\lb}{\lambda}
\newcommand{\la}{\lambda}

\newcommand{\Lb}{\Lambda}
\newcommand{\F}{\mathcal{F}}

\newcommand{\mean}[1]{\,-\hskip-1.08em\int_{#1}} 
\newcommand{\meantext}[1]{\,-\hskip-0.88em\int_{#1}} 

\newcommand{\au}{\boldsymbol u_\alpha}

\newcommand{\av}{\boldsymbol v_\alpha}

\def\ds{\displaystyle}
\def\eps{{\varepsilon}}
\def\N{\mathbb{N}}

\def\R{\mathbb{R}}

\def\F{\mathcal{F}}
\def\H{\mathcal{H}}

\def\Dr{D}

\newcommand{\be}{\begin{equation}}
\newcommand{\ee}{\end{equation}}
\newcommand{\bib}[4]{\bibitem{#1}{\sc#2: }{\it#3. }{#4.}}

\newcommand{\cp}{\mathop{\rm cap}\nolimits}
\newcommand{\ind}{\mathbbm{1}}

\numberwithin{equation}{section}
\theoremstyle{plain}
\newtheorem{teo}{Theorem}[section]
\newtheorem{lemma}[teo]{Lemma}
\newtheorem{cor}[teo]{Corollary}
\newtheorem{prop}[teo]{Proposition}
\newtheorem{deff}[teo]{Definition}
\newtheorem{oss}[teo]{Remark}

\title{Lipschitz regularity of the eigenfunctions on optimal domains}

\author{Dorin Bucur, Dario Mazzoleni, Aldo Pratelli, Bozhidar Velichkov}

\begin{document}

\maketitle

\begin{abstract}
We study the optimal sets $\Omega^\ast\subset\R^d$ for spectral functionals $F\big(\lambda_1(\Omega),\dots,\lambda_p(\Omega)\big)$, which are bi-Lipschitz with respect to each of the eigenvalues $\lambda_1(\Omega),\dots,\lambda_p(\Omega)$ of the Dirichlet Laplacian on $\Omega$, a prototype being the problem 
$$ \min{\big\{\la_1(\Om)+\dots+ \la_p(\Omega)\;:\;\Om\subset\R^d,\ |\Omega|=1\big\}}. $$ 
We prove the Lipschitz regularity of the eigenfunctions $u_1,\dots,u_p$ of the Dirichlet Laplacian on the optimal set $\Om^*$ and, as a corollary, we deduce that $\Om^*$ is open.
For functionals depending only on a generic subset of the spectrum, as for example $\lambda_k(\Omega)$ or $\lambda_{k_1}(\Omega)+\dots+\lambda_{k_p}(\Omega)$ , our result proves only the existence of a Lipschitz continuous eigenfunction in correspondence to each of the eigenvalues involved.
\end{abstract}

\vspace{1cm}

\section{Introduction}
In this paper we study the domains of prescribed volume, which are optimal for functionals depending on the spectrum of the Dirichlet Laplacian. Precisely, we consider shape optimization problems of the form
\begin{equation}\label{intro130930-1}
\min\Big\{F\big(\lambda_1(\Omega),\dots,\lambda_p(\Omega)\big)\,:\,\Omega\subset\R^d,\ |\Omega|=1\Big\},
\end{equation}
where $F:\R^p\to\R$ is a given continuous function, increasing in each variable, and $\lambda_k(\Omega)$, for $k=1,\dots,p$, denotes the $k^{th}$ eigenvalue of the Dirichlet Laplacian on $\Omega$, i.e. the $k^{th}$ element of the spectrum\footnote{We recall that due to the volume constraint $|\Omega|=1$ the Dirichlet Laplacian on $\Omega$ has compact resolvent and its spectrum is dicrete.} of the Dirichlet Laplacian.

The optimization problems of the form \eqref{intro130930-1} naturally arise in the study of physical phenomena as, for example, heat diffusion or wave propagation inside a domain $\Omega\subset\R^d$. Despite of their simple formulation, these problems turn out to be quite challenging and their analysis usually depends on sophisticated variational techniques. Even the question of the existence of a minimizer for the simplest spectral optimization problem  
\begin{equation}\label{bmpvi1}
\min\left\{\lambda_k(\Omega)\,:\ \Omega\subset\R^d,\ |\Omega|=1\right\},
\end{equation}
was answered only recently for general $k\in\N$ (see \cite{bulbk} and \cite{mp}).
This question was first formulated in the $19$th century by Lord Rayleigh in his treatise \emph{The Theory of Sound} \cite{rayleigh} and it was related to the specific case $k=1$.  It was proved only in the 1920s by Faber and Krahn that the minimizer in this case is the ball. From this result one can easily deduce the Krahn-Szeg\"o inequality, which states that a union of two disjoint balls is optimal for \eqref{bmpvi1} with $k=2$, i.e. it has the smaller second eigenvalue $\lambda_2$ among all sets of prescribed measure. An explicit construction of an optimal set for higher eigenvalues is an extremely difficult task. Balls are not always optimal, in fact it was proved by Keller and Wolf in 1994 (see \cite{kewo}) that a union of disjoint balls is not optimal for $\lambda_{13}$ in two dimensions. It was recently proved by Berger and Oudet\footnote{private communication} that the later result holds for all $k\in\N$ large enough, which confirmed the previous numerical results obtained in \cite{ou04} and \cite{anfr12}.

The classical variational approach of proving existence and regularity of minimizers failed to provide a solution to the spectral problems \eqref{intro130930-1} until the $1990$s, the main reason being the lack of an appropriate topology on the space of domains $\Omega\subset\R^d$. A suitable convergence called \emph{$\gamma$-convergence} was introduced by Dal Maso and Mosco (see \cite{dmmo86, dmmo87}) in the 1980s and was used by Buttazzo and Dal Maso (see \cite{budm93}) for proving in 1993 a very general existence result for \eqref{intro130930-1}, under the additional constraint $\Omega\subset\Dr$. The presence of the open bounded set $\Dr\subset\R^d$ as a geometric obstacle (a \emph{box}) provided the necessary compactness, needed to obtain the existence of an optimal domain in the class of \emph{quasi-open}\footnote{A quasi-open set is a level set $\{u>0\}$ of a Sobolev function $u\in H^1(\R^d)$} sets. The proof of existence of a quasi-open minimizer for \eqref{bmpvi1} and, more generally, of \eqref{intro130930-1} in the entire space $\R^d$ was concluded in 2011 with the independent results of Bucur (see \cite{bulbk}) and Mazzoleni and Pratelli (see \cite{mp}). Moreover, it was proved that the optimal sets are bounded (see \cite{bulbk} and \cite{mbound}) and of finite perimeter (see \cite{bulbk}).

The regularity of the optimal sets or the corresponding eigenfunctions turned out to be quite difficult question, due to the min-max nature of the spectral cost functionals, and is an open problem since the general Buttazzo-Dal Maso existence theorem. The only result that provides the complete regularity of the free boundary $\partial\Omega$ of the optimal set $\Omega$ concerns only the minimizers of \eqref{bmpvi1}\footnote{under the additional constraint $\Omega\subset\Dr$, where $\Dr\subset\R^d$ is a bounded open set.} in the special case $k=1$ and is due to Brian\c{c}on and Lamboley (\cite{brla}) who proved that the free boundary of the optimal sets is smooth. The implementation of this result for higher eigenvalues presents some major difficulties since the techniques, developed by Alt and Caffarelli in \cite{altcaf}, used in the proof are exclusive for functionals defined through a minimization and not min-max procedure on the Sobolev space $H^1_0(\Omega)$.

   In this paper we study the regularity of the eigenfunctions (or state functions) on the optimal set $\Omega^\ast$ for the problem \eqref{bmpvi1}. Our main tool is a result proved by Brian\c{c}on, Hayouni and Pierre (\cite{bhp05}), inspired by the pioneering work of  Alt and Caffarelli (see \cite{altcaf}) on the regularity for a free boundary problem. It states that a function $u\in H^1(\R^d)$, satisfying an elliptic PDE on the set $\Omega=\{|u|>0\}$, is Lipschitz continuous on the whole $\R^d$, if it satisfies the following \emph{quasi-minimality} property:
\begin{equation}\label{bmpvi50}
\int_{\R^d}|\nabla u|^2\,dx\le \int_{\R^d}|\nabla v|^2\,dx+cr^d,\qquad\forall v\in H^1(\R^d)\ \ \hbox{s.t.}\ u=v\ \hbox{on}\  \R^d\setminus B_r(x),
\end{equation}
for every ball $B_r(x)\subset\R^d$. 

   Since the variational characterization of the eigenvalue $\lambda_k$ is given through a min-max procedure the transfer of the minimality properties of $\Omega$ to an eigenfunction $u_k$ is a non-trivial task. In fact, it can be proved that the eigenfunction $u_k$ is a quasi-minimizer in the sense of \eqref{bmpvi50}, provided that the eigenvalue $\lambda_k(\Omega^\ast)$ is simple. Since the latter is expected not to be true in general, we use an approximation procedure with sets $\Omega_\eps$, which are solutions of a spectral optimization problems of the form 
$$\min\Big\{(1-\eps)\lambda_k(\Omega)+\eps\lambda_{k-1}(\Omega)+c|\Omega|:\ \Omega^\ast\subset\Omega\subset\R^d\Big\}.$$   
We study the Lipschitz continuity of the eigenfunctions $u_k^\eps$ on each $\Omega_\eps$ and then pass to the limit to recover the Lipschitz continuity of $u_k$ on $\Omega^\ast$ (see Theorem \ref{thk}). The uniformity of the Lipschitz constants is assured, roughly speaking, by the optimality condition on the free boundary of $\Omega_\eps$, which, in the case of regular $\Omega_\eps$ and simple eigenvalues, reads as
$$(1-\eps)|\nabla u_k^\eps|^2+\eps|\nabla u^\eps_{k-1}|^2=c.$$  
 
   The main result of the paper is Theorem \ref{thf.1}, which applies to shape supersolutions of functionals of the form
  $$\Omega\mapsto F\big(\lambda_{k_1}(\Omega),\dots,\lambda_{k_p}(\Omega)\big)+|\Om|,$$ 
where $F:\R^p\to\R$ is increasing and bi-Lipschitz in each variable.   Precisely, if  a set $\Om^*$ satisfies
$$F\big(\lb_{k_1}(\Omega^*), \dots, \lb_{k_p}(\Omega^*)\big)+|\Om^*|\le  F\big(\lb_{k_1}(\Omega), \dots, \lb_{k_p}(\Omega)\big)+ |\Om|, $$ 
for all measurable sets $\Om$ containing $\Om^*$,
then there exists a family of $L^2$-orthonormal eigenfunctions $u_{k_1},\dots,u_{k_p}$, corresponding respectively to $ \lambda_{k_1}(\Omega),\dots,\lambda_{k_p}(\Omega)$, which are Lipschitz continuous on $\R^d$.

In some particular cases, as for example linear combinations of the form
$$F\big( \lb_1(\Om),\dots, \lb_p(\Om)\big)=\sum_{i=1}^p \alpha_i \lb_i(\Om),$$
with strictly positive $\alpha_i$, for every $i=1, \dots, p$, the minimizers are moreover proved to be open sets (see Corollary \ref{funsum}), since in this case there exists an open set $\Om^{**}\subseteq\Om^*$ which has the same eigenvalues of $\Om^*$ up to order $p$.
For this easier case, in two dimensions, it is also possible to give a more direct proof which does not rely on the Alt-Caffarelli regularity techniques (see \cite{phdm}).


In \cite{bulbk}, the analysis of shape subsolutions gave some qualitative information on the optimal sets, in particular their boundedness and finiteness of the perimeter. Nevertheless, it is known that a subsolution may not be equivalent to an open set. Continuity of the state functions in free boundary problems relies, in general, on outer perturbations. Consequently the study of supersolutions became a fundamental target, which is partially attained in this paper.     In the case of subsolutions, the problem could be reduced to the analysis of a unique state function, precisely the torsion function, by controlling the variation of the $k^{th}$ eigenvalue for an inner geometric domain perturbation with the variation of the torsional rigidity. As far as we know, an analogous approach for the analysis of shape supersolutions can not be performed since one can not control the variation of the torsional rigidity by the variation of the $k^{th}$ eigenvalue. 

\medskip

This paper is organized as follows: in Section~\ref{prel} we recall some tools about Sobolev-like spaces, capacity and $\gamma$-convergence; in Section~\ref{quasimin} we deal with the Lipschitz regularity for quasi-minimizers of the Dirichlet energy and then, in Section~\ref{gap}, we apply these results to eigenfunctions of the Dirichlet Laplacian corresponding to a simple eigenvalue.
Then in Section~\ref{supersol} we introduce the notion of shape supersolution and we prove our main results Theorem \ref{thk} and Theorem \ref{thf.1}, concerning the Lipschitz regularity of the eigenfunctions associated to the general problem~\eqref{intro130930-1}. 
At last, in Section~\ref{openness}, we show that for some functionals we are able to prove that optimal sets are open.

\section{Preliminary results}\label{prel}
In what follows, we will use the following notations and conventions: 
\begin{itemize}
\item $C_d$ denotes a constant depending only on the dimension $d$ and if it is not specified it might change from line to line;
\item $\omega_d$ denotes the volume of the unit ball in $\R^d$ and thus $d\omega_d$ is the area of the unit sphere;
\item $\H^m$ denotes the $m$-dimensional Hausdorff measure in $\R^d$;
\item if the domain of integration is not specified, then it is assumed to be the whole space $\R^d$;
\item we denote the mean value of a function $u:\Omega\to\R$ with 
$$\mean{\Omega}u\,dx:=\frac{1}{|\Omega|}\int_\Omega u\,dx.$$ 
\end{itemize}

\subsection{Sobolev spaces and spectral minimizers}

Let $H\subset H^1(\R^d)$ be a closed linear subspace of $H^1(\R^d)$ such that the embedding $H\subset L^2(\R^d)$ is compact. We define the spectrum of the Laplace operator $-\Delta$ on $H$ as 
$$\lambda_k(H)=\min_{S_k}\max_{u\in S_k}\frac{\int|\nabla u|^2\,dx}{\int u^2\,dx},$$
where the minimum is over all $k$-dimensional subspaces $S_k$ of $H$. In the case when $H=H^1_0(\Omega)$, where $\Omega$ is an open set of finite measure, we use the usual notation $\lambda_k(\Omega):=\lambda_k(H^1_0(\Omega))$ and thus the $k^{th}$ eigenvalue of the Dirichlet Laplacian $\lambda_k$ on $\Omega$ can be seen as a functional on the open sets $\Omega\subset\R^d$. In this paper we are interested in the regularity of the eigenfunctions on the sets $\Omega$ which are minimal with respect to exterior perturbations, i.e.
\begin{equation}\label{mprbmpv}
F\big(\lambda_1(\Omega),\dots,\lambda_k(\Omega)\big)\le F\big(\lambda_1(\widetilde\Omega),\dots,\lambda_k(\widetilde\Omega)\big)+c|\widetilde\Omega\setminus\Omega|,\qquad\forall \widetilde\Omega\supset\Omega,
\end{equation}
where $F$ is a given function in $\R^k$, increasing in each variable. This is a property satisfied, for example, from the spectral minimizers, solution of the problem
\begin{equation}\label{minlbkbmpv}
\min\left\{\lambda_k(\Omega)+|\Om|:\ \Omega\subset\R^d\right\}.
\end{equation}
Since, at the moment, the problem \eqref{minlbkbmpv} is known to have solution only in the wider class of quasi-open sets (see \cite{budm93,bulbk,mp}), we extend the definition of $\lambda_k$ to a wider class of sets. Indeed, for every measurable $\Omega\subset\R^d$, we define the Sobolev space $H^1_0(\Omega)$ as
\begin{equation}\label{capsob}
H^1_0(\Omega)=\Big\{u\in H^1(\R^d):\ \cp(\{u\neq 0\}\setminus\Omega)=0\Big\},
\end{equation}
where for every $E\subset\R^d$ the capacity of $E$ is defined as
$$\cp(E)=\min\left\{\|v\|^2_{H^1(\R^d)}\;:\;v\in H^1(\R^d),\;v\geq 1 \mbox{  a.e. in a neighborhood of }E\right\}.$$
In the case when $\Omega$ is an open set the space defined in \eqref{capsob} coincides with the Sobolev space $H^1_0(\Om)$ defined as the closure of $C^\infty_c(\Omega)$ with respect to the norm $\|\cdot\|_{H^1}$. 
We say that the set $\Omega$ is quasi-open, if it is a level set $\Omega=\{u>0\}$ of a Sobolev function $u\in H^1(\R^d)$. In every measurable $\Omega$, there is a largest quasi-open set $\omega\subset\Omega$ (defined up to a set of zero capacity), which is also such that $H^1_0(\omega)=H^1_0(\Omega)$. 

In this paper we will deal mainly with the Sobolev-like spaces $\widetilde H^1_0(\Omega)$, defined for every measurable $\Omega\subset\R^d$ as 
\begin{equation}\label{meassob}
\widetilde H^1_0(\Omega)=\Big\{u\in H^1(\R^d):\ \big|\{u\neq 0\}\setminus\Omega\big|=0\Big\}.
\end{equation}
The inclusion $H^1_0(\Omega)\subset\widetilde H^1_0(\Omega)$ always holds, while the equality is achieved for open sets with Lipschitz boundary (see, for example, \cite{deve}) and it is not hard to construct open sets for which this equality is false (for example a ball minus one of its diameters). Thus we have $\lambda_k(\widetilde H^1_0(\Omega))\le \lambda_k(H^1_0(\Omega))$. Moreover, for every measurable $\Omega$ there is a largest quasi-open set $\omega$ such that $\omega\subset\Omega$ a.e. and $H^1_0(\omega)=\widetilde H^1_0(\omega)=\widetilde H^1_0(\Omega)$. Thus, for every set $\Omega$ satisfying \eqref{mprbmpv} there is a quasi-open set $\omega$ such that $\omega=\Omega$ a.e. and $H^1_0(\omega)=H^1_0(\Omega)$. Moreover, $\omega$ also satisfies \eqref{mprbmpv} with the functional $\lambda_k$ defined as 
\begin{equation}\label{lambdakdef}
\lambda_k(\Omega):=\lambda_k(\widetilde H^1_0(\Omega)),\ \hbox{for every}\ \Omega\subset\R^d. 
\end{equation}
From now on, we will use \eqref{lambdakdef} as a definition for $\lambda_k$. The main reason we use this definition is that, if the set of finite measure $\Omega$ satisfies \eqref{mprbmpv}, then for every $\eps>0$, $\Omega$ is the unique solution of
\begin{equation}\label{mprbmpv1}
\min\left\{F\big(\lambda_1(\widetilde\Omega),\dots,\lambda_k(\widetilde\Omega)\big)+(c+\eps)|\widetilde\Omega|:\ \widetilde\Omega\supset\Omega\right\}.
\end{equation}
Indeed, one can easily check that if $\Omega_1$ is a solution of \eqref{mprbmpv1}, then $|\Omega_1\Delta\Omega|=0$ and so, $\widetilde H^1_0(\Omega_1)=\widetilde H^1_0(\Omega)$.

\subsection{PDEs and eigenfunctions on measurable sets}
Let $\Omega\subset\R^d$ be a set of finite Lebesgue measure and $f\in L^2(\Omega)$. We say that the function $u$ satisfies the equation 
\begin{equation}\label{eleqbmpv}
-\Delta u=f,\qquad u\in\widetilde H^1_0(\Omega),
\end{equation}
if, for every $v\in \widetilde H^1_0(\Omega)$, we have 
$$\langle\Delta u+f,v\rangle:=-\int_{\R^d}\nabla u\cdot\nabla v\,dx+\int_{\R^d}fv\,dx=0,$$
or analogously, if $u\in \widetilde H^1_0(\Omega)$ is the minimizer of 
$$\min\left\{\frac12\int_\Omega |\nabla v|^2\,dx-\int_\Omega vf\,dx:\ v\in\widetilde H^1_0(\Omega)\right\}.$$

If $u$ is a solution of \eqref{eleqbmpv}, then there is a signed Radon measure $\mu$ such that for every $v\in H^1(\R^d)$, we have 
$$\langle\Delta u+f,v\rangle=\int_{\R^d} v\,d\mu.$$
In particular, $\Delta u$ is a signed Radon measure on $\R^d$. Indeed, if $u\ge 0$, then the functional $\Delta u+f$ (defined on $H^1(\R^d)$) is positive. Applying the Riesz's Theorem we obtain the existence of the measure $\mu$, which in this case is positive and finite on the compact sets (i.e. $\mu$ is a Radon measure). In the general case, we consider the functions $u^+=\sup\{u,0\}$ and $u^-=\sup\{-u,0\}$. Each of the functionals $\Delta u^++fI_{\{u>0\}}$ and $\Delta u^-+fI_{\{u<0\}}$ is positive and so, there are measures $\mu_1$ and $\mu_2$ such that $\mu_1=\Delta u^++fI_{\{u>0\}}$ and $\mu_2=\Delta u^-+fI_{\{u<0\}}$. Thus, we have that the signed measure $\mu=\mu_1-\mu_2$ is such that $\mu=\Delta u+f$. Moreover, we have that:
\begin{enumerate}
\item the support $supp(\mu)$ of $\mu$ is contained in the topological bondary $\partial\Omega$ of $\Omega$,
\item the measure $\mu$ is \emph{capacitary}, i.e. for each set $E$ of zero capacity, $\mu(E)=0$.
\end{enumerate}

 If $u\in H^1(\R^d)$ is a solution of \eqref{eleqbmpv} with $f\in L^\infty(\Omega)$, then every point $x_0\in\R^d$ is a Lebesgue point for $u$, i.e.
$$u(x_0)=\lim_{r\to0}\mean{B_r(x_0)}{u(x)\, dx}.$$ 
Moreover, the following result was proved in \cite{bhp05}.
\begin{prop}\label{bhpprop}
Suppose that $\Omega\subset\R^d$ is a set of finite Lebesgue measure and  $f\in L^\infty(\Omega)$. If $u\in \widetilde H^1_0(\Omega)$ is a solution of \eqref{eleqbmpv}, then for every $x\in\R^d$,
\begin{equation}
u(x)=\lim_{r\to0}\mean{\partial B_{r}(x)}{u\,d\H^{d-1}},
\end{equation}
and, for every $R>0$,
\begin{equation}\label{pwde4}
\mean{\partial B_{R}(x)}u\,d\H^{d-1}-u(x)=\frac{1}{d\omega_d}\int_{0}^{R} r^{1-d}\Delta u(B_r(x))\,dr.
\end{equation}
\end{prop}

\begin{oss}
We note that the above propositions applies to the eigenfunctions of the Dirichlet Laplacian on $\Omega$. Indeed, if $u_k$ is a solution of 
$$-\Delta u_k=\lambda_k(\Omega)\,u_k,\qquad u_k\in \widetilde H^1_0(\Omega),\ \int_\Omega u_k^2\,dx=1,$$
then we have the estimate (see \cite[Example 2.1.8]{davies})
\begin{equation}\label{davies218}
\|u_k\|_{L^\infty}\le e^{1/8\pi}\lambda_k(\Omega)^{d/4},
\end{equation}
and so $u_k$ satisfies the hypotheses of Proposition \ref{bhpprop}.
\end{oss}

Most of the perturbation techniques, that we will use in order to prove the Lipschitz continuity if the state functions $u$ on the optimal sets $\Omega$, provide us with information on the mean values $\meantext{B_r}u\,dx\ $ or $\meantext{\partial B_r}u\,d\H^{d-1}$. In order to transfer this information to the gradient $|\nabla u|$, we will need the following classical result.
 
\begin{oss}[Gradient estimate]\label{gradest}
Let $u\in H^1(B_r)$ is such that $-\Delta u=f$ in $B_r$ and $f\in L^\infty(B_r)$. Then we have 
\begin{equation}\label{gradest1}
\|\nabla u\|_{L^\infty(B_{r/2})}\le C_d\|f\|_{L^\infty(B_{r})}+\frac{2d}{r}\|u\|_{L^{\infty}(B_r)}.
\end{equation}
Note that, one can replace in \eqref{gradest1} $\|u\|_{L^{\infty}(B_r)}$ with the integral of $|u|$ on the boundary $\partial B_r$. In fact, we have the following estimate:
\begin{equation}
\|u\|_{L^\infty(B_{2r/3})}\le \frac{r^2}{2d}\|f\|_{L^\infty(B_r)}+ C_d\mean{\partial B_r}{|u|\,d\H^{d-1}}.
\end{equation}

Since, $\Delta u^++fI_{\{u>0\}}\ge0$ and $\Delta u^--fI_{\{u<0\}}\ge0$ in $B_r$, we have that $\Delta |u|+\|f\|_{L^\infty}\ge0$ in $B_r$.
Let $u_h$ be the harmonic function in $B_r$ with boundary values $u_h=|u|$ on $\partial B_r$. Since $|u|$ is non-negative, the same holds for $u_h$ and applying the Poisson's formula for the disk we have that 
\begin{equation}
\|u_h\|_{L^\infty(B_{2r/3})}\le C_d\mean{\partial B_r}{|u|\,d\H^{d-1}}.
\end{equation}
Moreover, by the maximum principle, we have that for any $x\in B_r$
\begin{equation}
\big| |u|-u_h \big|(x)\le \frac{r^2-|x|^2}{2d}\|f\|_{L^\infty(B_r)}.
\end{equation} 
Putting together the two estimates, we have:
\begin{align*}
\|u\|_{L^\infty(B_{2r/3})}&\le\frac{r^2}{2d}\|f\|_{L^\infty(B_r)}+\|u_h\|_{L^\infty(B_{2r/3})}\\
\\ 
&\le\frac{r^2}{2d}\|f\|_{L^\infty(B_r)}+ C_d\mean{\partial B_r}{|u|\,d\H^{d-1}}.
\end{align*}
\end{oss}

\subsection{The $\gamma$ and weak-$\gamma$ convergences}
In the proof of our main result (Theorem \ref{thk}) we will use a variational convergence defined on the measurable sets of finite Lebesgue measure. Indeed, for every $\Omega\subset\R^d$ with $|\Omega|<+\infty$, we will denote with $w_\Omega$ the solution of 
$$-\Delta w_\Omega=1,\qquad w_\Omega\in\widetilde H^1_0(\Omega).$$
We note that a measurable set $\Omega\subset\R^d$ is precisely determined, as a domain of the Sobolev space $\widetilde H^1_0(\Omega)$, by the \emph{energy function} $w_\Omega$. In fact, we have the equality 
$$\widetilde H^1_0(\Omega)=\widetilde H^1_0\big(\{w_\Omega>0\}\big).$$
If the measurable set $\Omega$ is such that $\big|\Omega\Delta\{w_\Omega>0\}\big|=0$, then we can choose its representative in the family of measurable set to be precisely the set $\{w_\Omega>0\}$.
   
\begin{deff}
We say that the sequence of sets of finite measure $\Omega_n$
\begin{itemize}
\item $\gamma$-\emph{converges} to the set $\Omega$, if the sequence $w_{\Omega_n}$ converges strongly in $L^2(\R^d)$ to the function $w_\Omega$;
\item \emph{weak-$\gamma$-converges} to the set $\Omega$, if the sequence $w_{\Omega_n}$ converges strongly in $L^2(\R^d)$ to the function $w\in H^1(\R^d)$ and $\Omega=\{w>0\}$.
\end{itemize}
%
\end{deff}
We note that in the case of a weak-$\gamma$-converging sequence $\Omega_n\to\Omega$, there is a comparison principle between the limit function $w=L^2-\lim_{n\to\infty}w_{\Omega_n}$ and the energy function $w_\Omega$. Indeed, we have the inequality $w\le w_\Omega$, which follows by the variational characterization of $w$, through the so called \emph{capacitary measures}, or it can also be proved directly by comparing the functions $w_{\Omega_n}$ to $w_\Omega$ (see \cite{buve}). Using only this weak maximum principle and the definitions above, one may deduce the following properties of the $\gamma$ and the weak-$\gamma$ convergences (for more details we refer the reader to the papers \cite{but10,budm93} and the books \cite{bubu05, hepi05}).
\begin{oss}[$\gamma$ and weak-$\gamma$-convergences]
If $\Omega_n$ $\gamma$-converges to $\Omega$, then it also weak-$\gamma$-converges to $\Omega$. Under the additional assumption $\Omega\subset\Omega_n$, for every $n\in\N$, we have that if $\Omega_n$ weak-$\gamma$-converges to $\Omega$, then $\Omega_n$ $\gamma$-converges to $\Omega$.
\end{oss}
\begin{oss}[measure and weak-$\gamma$-convergences]
If $\Omega_n$ converges to $\Omega$ in $L^1(\R^d)$, i.e. $\big|\Omega_n\Delta\Omega\big|\to0$, then up to a subsequence $\Omega_n$ weak-$\gamma$-converges to $\Omega$. On the other hand, if $\Omega_n$ weak-$\gamma$-converges to $\Omega$, then we have the following semi-continuity of the Lebesgue measure:
$$|\Omega|\le \liminf_{n\to\infty}|\Omega_n|.$$
\end{oss}
\begin{oss}[$\gamma$ and Mosco convergences]
\begin{enumerate}[(a)]
\item Suppose that $\Omega_n$ weak-$\gamma$-converges to $\Omega$. Then, if the sequence $u_n\in \widetilde H^1_0(\Omega_n)$ converges in $L^2(\R^d)$ to $u\in H^1(\R^d)$, we have that $u\in \widetilde H^1_0(\Omega)$. In particular, we obtain the semi-continuity of $\lambda_k$, with respect to the weak-$\gamma$-convergence:
$$\lambda_k(\Omega)\le\liminf_{n\to\infty}\lambda_k(\Omega_n).$$
\item Suppose that $\Omega_n$ $\gamma$-converges to $\Omega$. Then, for every $u\in \widetilde H^1_0(\Omega)$, there is a sequence $u_n\in \widetilde H^1_0(\Omega_n)$ converging to $u$ strongly in $H^1(\R^d)$. As a consequence, we have the continuity of $\lambda_k$ with respect to the $\gamma$-convergence:
$$\lambda_k(\Omega)=\liminf_{n\to\infty}\lambda_k(\Omega_n).$$
\end{enumerate}
\end{oss}
%

\section{Lipschitz continuity of energy quasi-minimizers}\label{quasimin}

In this section we study the properties of the local quasi-minimizers for the Dirichlet integral. More precisely, let $f\in L^2(\R^d)$ and let $u\in H^1(\R^d)$ satisfies
\begin{equation}
-\Delta u=f, \qquad u\in\widetilde H^1_0\big(\{u\neq0\}\big).
\end{equation}
\begin{deff}
We say that $u$ is a \emph{quasi-minimizer} for the functional 
\begin{equation}\label{defJf}
J_f(u)=\frac12\int_{\R^d}|\nabla u|^2\,dx-\int_{\R^d}uf\,dx,
\end{equation}
if there is a positive constant $C$ such that for every $r>0$ 
we have 
\begin{equation}\label{qminJ}
J_f(u)\le J_f(v)+Cr^d,\qquad \forall v\in \mathcal{A}_r(u),
\end{equation}  
where the admissible set $\mathcal{A}_r(u)$ is defined as
$$\mathcal{A}_r(u):=\Big\{v\in H^1(\R^d)\,:\, \exists x_0\in\R^d\ \hbox{such that}\ \ v-u\in H^1_0(B_r(x_0))\Big\}.$$
\end{deff}
\begin{deff}\label{defqmin}
We say that $u$ is a \emph{local quasi-minimizer}, if there are positive constants $\alpha$ and $r_0$ such that for every $0<r\le r_0$ we have
\begin{equation}
J_f(u)\le J_f(v)+Cr^d,\qquad \forall v\in \mathcal{A}_{r,\alpha}(u),
\end{equation}  
where the admissible set $\mathcal{A}_{r,\alpha}(u)$ is defined as
$$\mathcal{A}_{r,\alpha}(u):=\Big\{v\in H^1(\R^d)\,:\, \exists x_0\in\R^d\ \hbox{s.t.}\ \ v-u\in H^1_0(B_r(x_0)), \int |\nabla(u-v)|^2\,dx\le\alpha\Big\}.$$
\end{deff}
\begin{oss}
The local quasi-minimality condition is equivalent to suppose that for every ball $B_r(x_0)$, of radius smaller than $r_0$, and every $\varphi\in H^1_0(B_r(x_0))$, such that $\int|\nabla\varphi|^2\,dx\le \alpha$, we have 
\begin{equation}\label{qminequi1}
\vert\langle\Delta u+f,\varphi\rangle\vert\le \frac12\int_{B_r(x_0)}|\nabla\varphi|^2\,dx+Cr^d.
\end{equation}
Moreover, if for some constant $C>0$ $u$ satisfies 
\begin{equation}\label{qminequi2}
\vert\langle\Delta u+f,\varphi\rangle\vert\le C\left(\int_{B_r(x_0)}|\nabla\varphi|^2\,dx+r^d\right),
\end{equation}
for $r$ and $\varphi$, as above, then setting $\widetilde\varphi=(2C)^{-1}\varphi$, we have that $u$ satisfies \eqref{qminequi1} and so, is a quasi-minimizer.
\end{oss}

\begin{oss}
Let $\psi\in H^1_0(B_r(x_0))$. Testing \eqref{qminequi1} with $\varphi:=r^{d/2}\|\nabla \psi\|_{L^2}^{-1}\psi$, we obtain that the quasi-minimality of $u$ gives 
\begin{equation}\label{qminequi3}
\vert\langle\Delta u+f,\psi\rangle\vert\le Cr^{d/2}\left(\int_{B_r(x_0)}|\nabla\psi|^2\,dx\right)^{1/2}.
\end{equation}
Moreover, by the mean geometric-mean quadratic inequality, we have that condition \eqref{qminequi3} is equivalent to the quasi-minimality of $u$.
\end{oss}

\begin{oss}
If $f\in L^\infty(\R^d)$ and the support $\Omega$ is of finite Lebesgue measure, then the quasi-minimality of $u$ with respect to $J$ is equivalent to the quasi-minimality of $u$ with respect to the Dirichlet integral
$$J_0(u)=\int_{\R^d}|\nabla u|^2\,dx.$$ 
\end{oss}

In what follows we prove a Theorem concerning the Lipschitz continuity of the local quasi-minimizers. This result is a consequence of the techniques introduced by Brian\c{c}on, Hayouni and Pierre \cite{bhp05}.

\begin{teo}\label{main}
Let  $\Omega\subset\R^d$ be a measurable set of finite measure, $f\in L^\infty(\Omega)$ and the function $u\in H^1(\R^d)$ be such that
\begin{enumerate}[(a)]
\item $u$ is a solution of the following equation on $\Omega$
$$-\Delta u=f,\qquad u\in\widetilde H^{1}_0(\Omega);$$
\item $u$ is a local quasi-minimizer for the functional $J_f$, i.e. there are constants $r_0\le 1$ and $C_b$ such that for every $x\in\R^d$, every $0<r\le r_0$ and every $\varphi\in H^1_0(B_r(x))$ we have
\begin{equation}\label{b}
|\langle\Delta u+f,\varphi\rangle|\le C_b\|\nabla\varphi\|_{L^2}|B_r|^{1/2}.
\end{equation}
\end{enumerate} 
Then: 
\begin{enumerate}[(1)]
\item $u$ is Lipschitz continuous on $\R^d$ and the Lipschitz constant depends on $d$, $\|f\|_\infty$, $|\Om|$, $C_b$ and $r_0$.
\item the distribution $\Delta |u|$ is a Borel measure satisfying
\begin{equation}\label{cirm05}
|\Delta |u||(B_r(x))\le C\,r^{d-1},
\end{equation}
for every $x\in\R^d$ such that $u(x)=0$, where the constant $C$ depends on $d$, $\|f\|_\infty$, $|\Om|$ and $C_b$ (but not on $r_0$). 
\end{enumerate}
\end{teo}
A precise account on the Lipschitz constant of $u$ from Theorem \ref{main} is
$$\|u\|_{C^{0,1}}\le C_d\left(1+ |\Om|^ \frac{d+4}{2d}   +C_b + \frac{ |\Om|^ \frac{2}{d}}{r_0} \right)  \|f\|_\infty.$$
We notice that condition {\it (b)}  is also necessary for the Lipschitz continuity of $u$. In fact, it expresses in a weak form the boundedness of the gradient $|\nabla u|$ on the boundary $\partial\Omega$. 

The proof of this theorem is implicitly contained in \cite[Theorem 3.1]{bhp05}. For the sake of completeness, we reproduce it in the Appendix. 

\begin{teo}\label{cirm03}
Under the hypotheses of Theorem \ref{main}, assume that $u$ is a normalized eigenfunction  (i.e. there exists $\lambda >0$ such that $f=\lambda u$ and $\int u^2 dx =1$) satisfying condition (a) and (b). Then, the Lipschitz constant is  independent of $r_0$.
\end{teo}
\begin{proof}
We first notice that by \eqref{davies218} we have $\|f\|_\infty=\lambda\|u\|_\infty\le 2\lambda^{\frac{d+4}{4}}$.
By Theorem \ref{main}, applied to $u$ and $f:=\lambda u$, we have that $u$  is Lipschitz continuous. We shall prove that the Lipschitz constant is independent on $r_0$. We set $\widetilde\Omega:=\{u\neq0\}$ and we note that $\widetilde\Omega$ is an open set. Let $x$ be such that $d(x, \widetilde\Omega^c)< \min\{r_0/3,1
\}$  and let $y \in \partial \widetilde\Omega$ such that $R_x:=d(x, \widetilde\Omega^c)=|x-y|$. 
Following Remark \ref{gradest}

\begin{equation}\label{maine1131001a}
\begin{array}{lll}
\ds|\nabla u(x)|&\ds\le C_d\lambda \|u\|_{L^\infty}+\frac{2d}{R_x}\|u\|_{L^\infty(B_{R_x}(x))}\\
\\
&\ds\le C_d\lambda \|u\|_{L^\infty}+\frac{2d}{R_x}\|u\|_{L^\infty(B_{2R_x}(y))}\\
\\
&\ds\le (C_d+R_x)\lambda \|u\|_{L^\infty}+\frac{C_d}{R_x}\mean{\partial B_{3R_x}(y)}{|u|\,d\H^{d-1}}.
\end{array}
\end{equation} 
The last inequality comes from the estimate on $B_{2R_x}(y)$ of the function
$$|u(\cdot)|-\frac{(3R_x)^2 -|\cdot|^2}{2d} \lambda \|u\|_{\infty},$$
which is sub-harmonic on the ball $B_{3R_x}(y)$ (see Remark \ref{gradest}). Hence
\begin{equation}\label{maine1131001b}
\begin{array}{lll}
\ds|\nabla u(x)|&\ds\le (C_d+R_x)\lambda \|u\|_{L^\infty}+\frac{C_d}{R_x}\int_0^{3R_x}s^{1-d}|\Delta |u||(B_s(y))\,ds\\
\\
&\ds\le (C_d+R_x)\lambda \|u\|_{L^\infty}+3C_d C,
\end{array}
\end{equation} 
where $C$ is the constant from  \eqref{cirm05}.

Consider the function $P\in C^\infty(\widetilde\Omega)$ defined as
\begin{equation}
P:=|\nabla u|^2+\lambda u ^2-2\lambda ^2\|u \|^2_\infty w_{\widetilde\Omega},
\end{equation}
where $w_{\widetilde\Omega}$ is the solution of the equation 
\begin{equation*}
-\Delta w_{\widetilde\Omega}=1,\qquad w_{\widetilde\Omega}\in H^1_0(\widetilde\Omega).
\end{equation*}

A direct computation gives that $P$ is sub-harmonic on the open set $\widetilde \Omega$, i.e.
\begin{equation}
\Delta P=\big(2[Hess(u )]^2-2\lambda |\nabla u |^2\big)+\big(2\lambda |\nabla u |^2-2\lambda ^2u ^2\big)+2\lambda ^2\|u \|^2_\infty\ge0.
\end{equation}
Thus, by the maximum principle we get 
$$\sup\big\{P(x)\,:\,x\in\widetilde\Omega\big\}\le\sup\big\{P(x)\,:\, x \in \widetilde\Omega,\ \hbox{dist}(x, \partial \widetilde\Omega) < r_0/3 \big\},$$
and so, using the boundary estimate \eqref{maine1131001b}, we obtain
\begin{equation}\label{the8}
\begin{array}{lll}
\|\nabla u \|_\infty^2&\le 2\lambda ^2\|u\|^2_\infty \|w_{\widetilde\Omega}\|_\infty+2\lambda \|u\|^2_\infty+\big((C_d+1)\lambda \|u\|_{L^\infty}+3C_d C\big)^2.
\end{array}
\end{equation}
Now the conclusion follows by \eqref{davies218} and the classical bound $\|w_{\widetilde\Omega}\|_\infty\le C_d |\widetilde\Omega|^{2/d}$ (see, for example, \cite[Theorem 1]{talenti}). 

\end{proof}
\begin{oss}\label{bvdp01}
Notice that the Lipschitz norm of $u$ satisfying the hypotheses of Theorem \ref{cirm03}, depends ultimately  on $d$, $|\Omega|$ and $\lambda$.
\end{oss}

\section{Shape quasi-minimizers for Dirichlet eigenvalues}\label{gap}
In this section we discuss the regularity of the eigenfunctions on sets which are minimal with respect to a given (spectral) shape functional.
In what follows we denote with $\mathcal{A}$ the family of subset of $\R^d$ with finite Lebesgue measure endowed with the equivalence relation $\Om \sim \tilde \Om$, whenever $|\Om\Delta \tilde \Om|=0$.

\begin{deff}\label{qmin}
We say that the measurable set $\Omega\in  {\mathcal A}$ is a shape quasi-minimizer for the functional $\mathcal F:\mathcal{A}\to\R$, if there exist constants $C>0$ and $r_0>0$ such that every $0<r\le r_0$ we have 
$$\mathcal F(\Omega)\le \mathcal F(\tilde \Omega)+ C|B_r|,\qquad \forall \widetilde\Omega\in\mathcal{A}_r(\Omega),$$
where the admissible set of perturbations $\mathcal{A}_r(\Omega)$ is given by
$$\mathcal{A}_r(\Omega)=\Big\{\widetilde\Omega\in\mathcal{A}\,:\,\exists x\in\R^d\ \hbox{such that}\quad \Omega \Delta \widetilde \Omega \subseteq B_r(x)\Big\}.$$
\end{deff}

\begin{oss}\label{cham1}
If the functional $\mathcal F$ is non-increasing with respect to inclusions, then $\Om$ is a shape quasi-minimizer, if and only if,
$$\mathcal F(\Omega)\le \mathcal F\big( \Omega\cup B_r(x)\big)+ C|B_r|.$$
\end{oss}

We expect that the property of shape quasi-minimality contains some information on the regularity of $\Omega$. In fact, for some shape functionals $\mathcal{F}$ one can easily deduce from the shape quasi-minimality of $\Omega$ the quasi-minimality of the state functions on $\Omega$. For example, suppose that $\Omega$ is a shape quasi-minimizer for the Dirichlet Energy
$$E(\Omega):=\min\Big\{J_1(u):\ u\in \widetilde H^1_0(\Omega)\Big\}.$$
Then, for $r>0$ small enough and $\widetilde \Omega\Delta\Omega \subset B_r(x)$, we have 
$$J_1(w_\Omega)=E(\Omega)\le E(\widetilde\Omega)+C|B_r|\le J_1(w_\Omega+\varphi)+C|B_r|,$$
where $w_\Omega\in\widetilde H^1_0(\Omega)$ is the energy function on $\Omega$ and   $\varphi$ is any function from $H^1_0(B_r)$.
Thus the function $w_\Omega$ is a quasi-minimizer for the functional $J_1$ in sense of Definition \ref{defqmin} and so, by Theorem \ref{main}, we can conclude that the energy function $w_\Omega$ is Lipschitz continuous on $\R^d$.

The case $\mathcal F=\lambda_k$ is more involved, since the $k^{th}$ eigenvalue is not defined through a single state function, but is variationally characterized by a min-max procedure involving an entire linear subspace of $\widetilde H^1_0(\Omega)$. In order to transfer the minimality information from $\Omega$ to its eigenfunctions $u_k$, we need an estimate on the variation of $\lambda_k$, with respect to external perturbation, in terms of the variation of the energy of $u_k$.

   In Lemma \ref{measl} below, we assume that $\Omega$ is a generic set of finite measure and $l\ge 1$ is such that  
\begin{equation}\label{cirm01}
\lambda_{k}(\Omega)=\dots=\lambda_{k-l+1}(\Omega)>\lambda_{k-l}(\Omega).
\end{equation}
We also choose $ u_{k-l+1},\dots, u_k$ to be $l$ normalized orthogonal eigenfunctions corresponding to $k^{th}$ eigenvalue $\lambda_k(\Omega)$ of the Dirichlet Laplacian on $\Omega$.

The following notation is used: given a vector $\alpha= (\alpha_{k-l+1},...,\alpha_k)\in \R^l$, we denote with
${\boldsymbol u}_\alpha$ the corresponding linear combination 
\begin{equation}\label{cirm02}
{\boldsymbol  u}_\alpha= \alpha_{k-l+1}u_{k-l+1}+...+\alpha_{k}u_k.
\end{equation}

\begin{lemma}\label{measl}
Let $\Omega\subset\R^d$ be a set of finite measure and $l\ge 1$ is such that \eqref{cirm01} holds. Then there is a constant $r_0>0$ such that for every $x\in\R^d$, every $0<r<r_0$ and every $l$-uple of functions $v_{k-l+1},\dots,v_{k}\in H^1_0(B_r(x))$ with $\int|\nabla v_j|^2\le 1$, for $j=k-l+1,\dots,k$, there is a unit vector $\alpha \in\R^l$ such that
\begin{equation}\label{l1}
\lambda_k(\Omega\cup B_r(x))\le\frac{\int |\nabla(\boldsymbol u_\alpha+\boldsymbol v_\alpha)|^2\,dx+(\lambda_{k-l}(\Omega)+1)\int |\nabla \boldsymbol v_\alpha|^2\,dx}{\int |\boldsymbol u_\alpha+\boldsymbol v_\alpha|^2\,dx-\frac12\int |\nabla \boldsymbol v_\alpha|^2\,dx},
\end{equation}
where $ \boldsymbol u_\alpha,  \boldsymbol v_\alpha$ are defined using notation \eqref{cirm02}.

 The constant $r_0$ depends on $\Omega$. In particular, if the gap $\lambda_{k-l+1}(\Omega)-\lambda_{k-l}(\Omega)$ vanishes, $r_0$ vanishes as well. 
\end{lemma}

\begin{proof} Without loss of generality, we can suppose $x=0$.
By the definition of the $k^{th}$ eigenvalue, we know that
$$\lb_k(\Om \cup B_r)\le \max \left\{\frac{\int|\nabla u|^2 dx}{\int u^2 dx} : u \in \mbox{span}\big\langle u_1,..., u_{k-l}, u_{k-l+1}+ v_{k-l+1}, ..., u_{k}+ v_{k}\big\rangle\right\}.$$
The maximum is attained for a linear combination 
$$\alpha_1 u_1+...+\alpha_{k-l} u_{k-l}+ \alpha_{k-l+1} (u_{k-l+1}+ v_{k-l+1})+ ...+  \alpha_{k} (u_{k}+ v_{k}).$$
Note that if $\lb_{k-l}(\Om) < \lb_k(\Om\cup B_r)$, then the vector 
$$\alpha=(\alpha_{k-l+1},...,\alpha_k)\in \R^l,$$
is non zero, and moreover can be chosen to be unitary. 
The inequality $\lb_{k-l}(\Om) < \lb_k(\Om\cup B_r(x))$, is true for every $x$ and every $r < r_0$ provided $r_0$ is small enough. This can be proved for instance by contradiction, since for every $x_n\in \R^d$ and for every $r_n\ra 0$, we have that $\Om\cup B_{r_n}(x_n)$ $\g$-converges to $\Om$. 

For simplicity, we denote $\lb_j=\lb_j(\Om)$, for every $j$. 

Using the notation \eqref{cirm02}, 
 for $r_0$ small enough, we have
\begin{equation}\label{l3}
\begin{array}{ll}
\ds \frac{\int|\nabla (\au+\av)|^2\,dx}{\int|\au+\av|^2\,dx}&=\ds \frac{\lambda_k+2\int \nabla \au\cdot\nabla \av\,dx+\int|\nabla \av|^2\,dx}{1+2\int \au \av\,dx+\int \av^2\,dx}\\
\\
&\ds \ge\frac{\lambda_k-2\left(\int_{B_r}|\nabla \au|^2\,dx\right)^{1/2}\left(\int_{B_r} |\nabla \av|^2\,dx\right)^{1/2}}{1+2\left(\int_{B_r} \au^2\,dx\right)^{1/2}\left(\int_{B_r} \av^2\,dx\right)^{1/2}+\int \av^2\,dx}\\
\\
&\ds \ge\frac{\lambda_k-2\left(\int_{B_{r_0}}|\nabla \au|^2\,dx\right)^{1/2}\left(\int_{B_r} |\nabla \av|^2\,dx\right)^{1/2}}{1+2\left(\int_{B_r} \av^2\,dx\right)^{1/2}+\int_{B_r} \av^2\,dx}\\
\\
&\ds \ge\frac{\lambda_k-2\left(\int_{B_{r_0}}|\nabla \au|^2\,dx\right)^{1/2}\left(\int_{B_r} |\nabla \av|^2\,dx\right)^{1/2}}{1+2C_d|B_{r_0}|^{1/d}\left(\int_{B_r} |\nabla \av|^2\,dx\right)^{1/2}+|C_dB_{r_0}|^{2/d}\int_{B_r} |\nabla \av|^2\,dx}\\
\\
&\ds \ge\frac{\lambda_{k-l}+\lambda_k}{2}.
\end{array}
\end{equation}

If all $\alpha_i$ for $i=1,..,k-l$ are zero, then the assertion of the theorem is trivially true. Otherwise, we define
$$u= \frac{1}{\sqrt{\alpha_1^2+...+\alpha_{k-l}^2}} (\alpha_1 u_1+...+\alpha_{k-l} u_{k-l}).$$
So $\int u^2 =1$ and $\int|\nabla u|^2 \le \lb_{k-l}$.

Consequently, 
$$\lb_k(\Om \cup B_r)\le \max \left\{\frac{\int|\nabla ( \au+\av +tu)|^2 dx}{\int |\au+\av+tu|^2 dx} : t \in \R\right\}.$$

We have 
\begin{equation}
\begin{array}{c}
\ds \frac{\int|\nabla (tu+\au+\av)|^2\,dx}{\int(tu+\au+\av)^2\,dx}
\le\frac{t^2\lambda_{k-l}(\Omega)+2t\int\nabla u\cdot\nabla(\au+\av )\,dx+\int|\nabla (\au+\av)|^2\,dx}{t^2+2t\int u(\au+\av)\,dx+\int|\au+\av|^2\,dx}\\
\\
\ds  =\frac{t^2\lambda_{k-l}(\Omega)+2t\int\nabla u\cdot\nabla \au \,dx+2t\int_{B_r}\nabla u\cdot\nabla \av\,dx+\int_{\R^d}|\nabla( \au+\av)|^2dx}{t^2+2t\int u \au\,dx+2t\int_{B_r}u\av\,dx+\int_{\R^d}|\au+\av|^2\,dx}\\
\\
\ds  =\frac{t^2\lambda_{k-l}(\Omega)+2t\int_{B_r}\nabla u\cdot\nabla \av\,dx+\int_{\R^d}|\nabla(\au+\av)|^2dx}{t^2+2t\int_{B_r}u\av\,dx+\int|\au+\av|^2dx}:= F(t).
\end{array}
\end{equation} 

For sake of simplicity we pose:

\begin{equation}\label{ab}
\begin{array}{ll}
a=\int_{B_r}\nabla u\cdot\nabla \av\,dx,&\qquad b=\int_{B_r}u\av\,dx,\\
\\
A=\int_{\R^d}|\nabla (\au+\av)|^2\,dx,&\qquad B=\int_{\R^d}| \au+\av|^2\,dx.
\end{array}\end{equation}
Note that we can make $a$ and $b$ arbitrarily small, by choosing $r_0$ small enough. In fact, we have the following estimates:
\begin{equation}\label{ab2}
|a|\le\left(\int_{B_r}|\nabla u|^2\,dx\right)^{1/2}\left(\int_{B_r}|\nabla \av|^2\,dx\right)^{1/2}\le\left(\int_{B_{r_0}}|\nabla u|^2\,dx\right)^{1/2}\left(\int_{B_r}|\nabla \av|^2\,dx\right)^{1/2},
\end{equation}
\begin{equation}\label{ab3}
|b|\le\left(\int_{B_r} u^2\,dx\right)^{1/2}\left(\int_{B_r}\av^2\,dx\right)^{1/2}\le C_d r_0\left(\int_{B_r}|\nabla \av|^2\,dx\right)^{1/2}.
\end{equation}

 Moreover, we can suppose that 
 $$\lambda_k/2\le A\le 2\lambda_k+1,\qquad 1/2\le B\le 2.$$
 
By \eqref{l3} and the fact that $\lim_{t\to\pm\infty}F(t)=\lambda_{k-l}<\frac{\lambda_{k-l}+\lambda_k}{2}\le F(0)$, we have that the maximum of $F$ is attained in $\R$. Computing the derivative, the zeros $t$ of  $F'$ satisfy
$$(\lambda_{k-l} t+a)(t^2+2bt+B)-(t+b)(\lambda_{k-l} t^2+2at+A)=0,$$
or, after simplification,
$$t^2(\lambda_{k-l} b-a)+t(\lambda_{k-l} B-A)+(aB-bA)=0.$$
Thus, we have that $\|F\|_\infty=\max\left\{F(t_1),F(t_2)\right\}$, where
\begin{equation}
\begin{array}{ll}
t_{1,2}&=\ds \frac{A-\lambda_{k-l} B\pm\sqrt{(A-\lambda_{k-l} B)^2-4(\lambda_{k-l} b-a)(aB-bA)}}{2(\lambda_{k-l} b-a)}\\
\\
&=\ds \frac{A-\lambda_{k-l} B}{2(\lambda_{k-l} b-a)}\left(1\pm\sqrt{1-\frac{4(\lambda_{k-l} b-a)(aB-bA)}{(A-\lambda_{k-l} B)^2}}\right)
\end{array}
\end{equation}
We choose $r_0$ small enough, in order to have 
$$\left|\frac{4(\lambda_{k-l} b-a)(aB-bA)}{(A-\lambda_{k-l} B)^2}\right|<\frac{1}{2}.$$
Then, since the function $x\mapsto \sqrt{1-x}$ is bounded and $1$-Lipschitz on the interval $(-\frac{1}{2},\frac{1}{2})$, we have the following estimate
\begin{equation}\label{l6}
\begin{array}{c}
\ds  |t_1|=\ds  \left|\frac{A-\lambda_{k-l} B}{2(\lambda_{k-l} b-a)}\left(1-\sqrt{1-\frac{4(\lambda_{k-l} b-a)(aB-bA)}{(A-\lambda_{k-l} B)^2}}\right)\right| \\
\le\ds \left|\frac{A-\lambda_{k-l} B}{2(\lambda_{k-l} b-a)}\right|\cdot\left|\frac{4(\lambda_{k-l} b-a)(aB-bA)}{(A-\lambda_{k-l} B)^2}\right|\\
\\
\ds  \le 2\left|\frac{aB-bA}{A-\lambda_{k-l} B}\right|\le 2\frac{|a|B+|b|A}{A-\lambda_{k-l} B}\le 4\frac{|a|+\lambda_k|b|}{A-\lambda_{k-l} B}\\
\\
\ds  \le 16\frac{|a|+\lambda_k|b|}{\lambda_k-\lambda_{k-l}}\le \left(\int_{B_r}|\nabla \av|^2dx\right)^{1/2}.
\end{array}
\end{equation}
The last inequality is obtained using \eqref{ab2} and \eqref{ab3}, for $r_0$ small enough. On the other hand, for $t_2$, we have
\begin{equation}\label{l7}
\begin{array}{ll}
\ds  \frac{1}{2}\left|\frac{A-\lambda_{k-l} B}{\lambda_{k-l} b-a}\right|\le|t_2|\le2\left|\frac{A-\lambda_{k-l} B}{\lambda_{k-l} b-a}\right|.
\end{array}
\end{equation}

Note that if we chooose $r_0$ such that $|t_1|<|t_2|$, then the maximum cannot be attained in $t_2$. In fact, $(\lambda_{k-l}b-a)t_2>0$ and so, in $t_2$, the derivative $F'$ changes sign from negative to positive, if $t_2>0$ and from negative to positive, if $t_2<0$, which proves that the maximum is attained in $t_1$. Choosing $r_0$ such that 
$$|a|\le\frac12\left(\int_{B_r}|\nabla \av|^2dx\right)^{1/2},\qquad |b|\le\frac14\left(\int_{B_r}|\nabla \av|^2dx\right)^{1/2},$$
we have

\begin{equation}
\begin{array}{ll}
F(t_1)&\ds \le\frac{\lambda_{k-l} t_1^2+2at_1+A}{t_1^2+2bt_1+B}\le\frac{\lambda_{k-l} t_1^2+|2at_1|+A}{t_1^2-|2bt_1|+B}\\
\\
&\ds \le\frac{\lambda_{k-l}\int_{B_r}|\nabla \av|^2dx+2|a|\left(\int_{B_r}|\nabla \av|^2dx\right)^{1/2}+A}{B-2|b|\left(\int_{B_r}|\nabla \av|^2dx\right)^{1/2}}\\
\\
&\ds \le\frac{A+\left(\lambda_{k-l}+1\right)\int_{B_r}|\nabla \av|^2dx}{B-\frac12\int_{B_r}|\nabla \av|^2dx},
\end{array}
\end{equation}
and so, the conclusion.
\end{proof}

\begin{oss}
The preceding Lemma \ref{measl} points out the main difficulty in the study of the regularity of spectral minimizers. Indeed, let $\Omega^\ast$ be a solution of a spectral optimization problem of the form \eqref{intro130930-1} involving $\lambda_k$ and
such that \eqref{cirm01} holds for some $l>1$. Then every perturbation $\widetilde u_k=u_k+v$ of the eigenfunction $u_k\in \widetilde H^1_0(\Omega^\ast)$ gives information on a linear combination $\au$ of eigenfunctions $u_k,\dots,u_{k-l+1}$, instead on the function $u_k$. Recovering some information on $u_k$ from an estimate on the linear combination is a difficult task since the combination itself depends on the perturbation $v$.
\end{oss}

\begin{oss}\label{meas2}
In case $\lb_k(\Om)>\lb_{k-1}(\Om)$, the result of the lemma above, states as
\begin{equation}\label{cham}
\lambda_k(\Omega\cup B_r(x))\le\frac{\int |\nabla(u_k+v)|^2\,dx+(\lambda_{k-1}(\Omega)+1)\int |\nabla v|^2\,dx}{\int |u_k+v|^2\,dx-\frac12\int |\nabla v|^2\,dx},
\end{equation}
for every $r<r_0$ and every $v \in H^1_0(B_r(x))$ such that $\int|\nabla v|^2\,dx\le 1$. 
\end{oss}

\begin{lemma}\label{double3}
Let $\Omega\subset\R^d$ be a shape quasi-minimizer for $\lambda_k$ such that $\lambda_k(\Omega)>\lambda_{k-1}(\Omega)$. Then every eigenfunction $u_k\in\widetilde H^1_0(\Omega)$, normalized in $L^2$ and corresponding to the eigenvalue $\lambda_k(\Omega)$, is Lipschitz continuous on $\R^d$. 
\end{lemma}
\begin{proof}
Let $u_k$ be a normalized eigenfunction corresponding to $\lambda_k$. By the shape quasi-minimality of $\Omega$, we have
\begin{equation}
\lambda_k(\Omega)\le\lambda_k\big(\Omega\cup B_r(x)\big)+C|B_r|.
\end{equation}
Applying the estimate \eqref{cham} for $v\in H^1_0(B_r)$, we obtain

\begin{equation}\label{cirm07}
\left|\langle\Delta u_{k}+\lambda_k(\Omega)u_{k},v\rangle\right|\le C|B_r|+\big(\lambda_{k}(\Omega)+1\big)\int|\nabla v|^2\,dx,
\end{equation}
and so, the function $u_k$ is a quasi-minimizer for the functional
$$u\mapsto \frac12\int_{\R^d}|\nabla u|^2\,dx-\int_{\R^d}\lambda_k(\Omega)u_k u\,dx.$$
%
%
Since $u_k$ is bounded by \eqref{davies218}, the claim follows by Theorem \ref{cirm03}.
\end{proof}

\section{Shape supersolutions of spectral functionals}\label{supersol}
\begin{deff}
We say that the set $\Omega\subset\R^d$ is a \emph{shape supersolution} for the functional $\F:\mathcal{A}\to\R$, defined on the class of Lebesgue measurable sets $\mathcal{A}$, if $\Omega$ satisfies 
$$\mathcal{F}(\Omega)\le \mathcal{F}(\widetilde\Omega),\qquad \forall \widetilde\Omega\supset\Omega.$$
\end{deff}

\begin{oss}\label{bvdp02}
\begin{itemize}
\item If $\Om^*$ is a shape supersolution for $\mathcal{F}+\Lambda|\cdot|$, for some $\Lambda>0$, then for every $\Lb' > \Lb$ the set $\Om^*$
 is the unique solution of 
 $$\min\Big\{\mathcal{F}(\Om)+\Lb'|\Om| :\ \Om\ \hbox{Lebesgue measurable}, \;\; \Om \supset \Om^*\Big\}.$$

\item If the functional $\mathcal{F}$ is non-increasing with respect to the inclusion, we have, by Remark \ref{cham1}, that every shape supersolution $\Omega$ of $\mathcal{F}+\Lambda|\cdot|$, where $\Lambda>0$, is also a shape quasi-minimizer.
\end{itemize}
\end{oss}

In Lemma \ref{double3} we showed that the $k^{th}$ eigenfunctions of the the shape quasi-minimizers for $\lambda_k$ are Lipschitz continuous under the assumption $\lambda_k(\Omega)>\lambda_{k-1}(\Omega)$. In the next Theorem, we show that for shape supersolutions of $\lambda_k+\Lambda|\cdot|$ the later assumption can be dropped. 

\begin{teo}\label{thk}
Let $\Omega^\ast\subset\R^d$ be a bounded shape supersolution for the functional $\lambda_k+\Lambda|\cdot|$, for some $\Lambda>0$. Then there is an eigenfunction $u_k\in\widetilde H^1_0(\Omega^\ast)$, normalized in $L^2$ and corresponding to the eigenvalue $\lambda_k(\Omega^\ast)$, which is Lipschitz continuous on $\R^d$. 
\end{teo}
\begin{proof}
We first note that if $\lambda_k(\Omega^\ast)>\lambda_{k-1}(\Omega^\ast)$, then the claim follows by Lemma \ref{double3}. Suppose now that $\lambda_k(\Omega^\ast)=\lambda_{k-1}(\Omega^\ast)$. For every $\eps\in(0,1)$ consider the problem
\begin{equation}\label{thk1}
\min\Big\{(1-\eps)\lambda_k(\Omega)+\eps\lambda_{k-1}(\Omega)+2\Lambda|\Omega|:\ \Omega\supset\Omega^\ast\Big\}.
\end{equation} 
We consider the following two cases:
\begin{enumerate}[(i)]
\item Suppose that there is a sequence $\eps_n\to0$ and a sequence $\Omega_{\eps_n}$ of corresponding minimizers for \eqref{thk1} such that $\lambda_k(\Omega_{\eps_n})>\lambda_{k-1}(\Omega_{\eps_n})$. For each $n\in\N$, $\Omega_{\eps_n}$ is a shape supersolution for the functional $\lambda_k+2(1-\eps_n)^{-1}\Lambda|\cdot|$ and so, by Lemma \ref{double3}, we have that for each $n\in\N$ the normalized eigenfunctions $u_k^{n}\in \widetilde H^1_0(\Omega_{\eps_n})$, corresponding to $\lambda_k(\Omega_{\eps_n})$, are Lipschitz continuous on $\R^d$.  We now prove that $\Omega_{\eps_n}$ $\gamma$-converges to $\Omega^*$ as $n\to\infty$. Indeed, by \cite[Proposition 5.12]{bubuve}, $\Omega_{\eps_n}$ are all contained in some ball $B_R$ with $R$ big enough. Thus, there is a weak-$\gamma$-convergent subsequence of $\Omega_{\eps_n}$ and let $\widetilde\Omega$ be its limit. Then $\widetilde\Omega$ is a solution of the problem 
\begin{equation}\label{thk2}
\min\Big\{\lambda_k(\Omega)+2\Lambda|\Omega|:\ \Omega\supset\Omega^\ast\Big\}.
\end{equation}   
On the other hand, by Remark \ref{bvdp02} we have that $\Omega^\ast$ is the unique solution of \eqref{thk2} and so, $\widetilde\Omega=\Omega^\ast$. Since the weak $\g$-limit $\Om^*$ satisfies $ \Omega^\ast\subset\Omega_{\eps_n}$ for every $n\in\N$, then $\Omega_{\eps_n}$ $\gamma$-converges to $\Om^*$. By the metrizability of the $\gamma$-convergence, we have that $\Omega^\ast$ is the $\gamma$-limit of $\Omega_{\eps_n}$ as $n\to\infty$. As a consequence, we have that $\lambda_k(\Omega_{\eps_n})\to\lambda_k(\Omega^\ast)$ and by Remark \ref{bvdp01} we have that the sequence $u_k^{n}$ is uniformly Lipschitz.

Then, we can suppose that, up to a subsequence $u_k^n\to u$ uniformly and weakly in $H^1_0(B_R)$, for some $u\in H^1_0(B_R)$, Lipschitz continuous on $\R^d$. By the weak convergence of $u_k^n$, we have that for each $v\in H^1_0(\Omega^\ast)$ 
$$\int \nabla u\cdot\nabla v\,dx=\lim_{n\to\infty}\int\nabla u_k^n\cdot\nabla v\,dx=\lim_{n\to\infty}\lambda_k(\Omega_{\eps_n})\int u_k^n v\,dx=\lambda_k(\Omega^\ast)\int u v\,dx.$$
By the $\gamma$-convergence of $\Omega_{\eps_n}$, we have that $u\in H^1_0(\Omega^\ast)$ and so $u$ is a $k^{th}$ eigenfunction of the Dirichlet Laplacian on $\Omega^\ast$. 
\item Suppose that there is some $\eps_0\in(0,1)$ such that $\Omega_{\eps_0}$ is a solution of \eqref{thk1}  and $\lambda_k(\Omega_{\eps_0})=\lambda_{k-1}(\Omega_{\eps_0})$. Then, $\Omega_{\eps_0}$ is also a solution of \eqref{thk2} and, by Remark \ref{bvdp02},  $\Omega_{\eps_0}=\Omega^\ast$. Thus we obtain that $\Omega^\ast$ is a shape supersolution for $\lambda_{k-1}+2\eps_0^{-1}\Lambda|\cdot|$. If we have 
$$\lambda_k(\Omega^\ast)=\lambda_{k-1}(\Omega^\ast)>\lambda_{k-2}(\Omega^\ast),$$
then, we apply Lemma \ref{double3} obtaining that each eigenfunction corresponding to $\lambda_{k-1}(\Omega^\ast)$ is Lipschitz continuous on $\R^d$. On the other hand, if 
$$\lambda_k(\Omega^\ast)=\lambda_{k-1}(\Omega^\ast)=\lambda_{k-2}(\Omega^\ast),$$
then we consider, for each $\eps\in(0,1)$, the problem
\begin{equation}\label{thk4}
\min\Big\{(1-\eps_0)\lambda_k(\Omega)+\eps_0\big[(1-\eps)\lambda_{k-1}(\Omega)+\eps\lambda_{k-2}(\Omega)\big]+3\Lambda|\Omega|:\ \Omega\supset\Omega^\ast\Big\}.
\end{equation} 
One of the following two situations may occur:
\begin{enumerate}[(a)]
\item There is a sequence $\eps_n\to0$ and a corresponding sequence $\Omega_{\eps_n}$ of minimizers of \eqref{thk4} such that 
$$\lambda_{k-1}(\Omega_{\eps_n})>\lambda_{k-2}(\Omega_{\eps_n}).$$
\item There is some $\eps_1\in(0,1)$ and $\Omega_{\eps_1}$, solution of \eqref{thk4}, such that
$$\lambda_{k-1}(\Omega_{\eps_1})=\lambda_{k-2}(\Omega_{\eps_1}).$$
\end{enumerate}
If the case $(a)$ occurs, then since $\Omega_{\eps_n}$ is a shape quasi-minimizer for $\lambda_{k-1}$, by Lemma \ref{double3} we obtain the Lipschitz continuity of the eigenfunctions $u_{k-1}^{n}$, corresponding to $\lambda_{k-1}$ on $\Omega_{\eps_n}$. Repeating the argument from $(i)$, we obtain that $\Omega_{\eps_n}$ $\gamma$-converges to $\Omega^\ast$ and that the sequence of eigenfunctions $u_{k-1}^{n}\in H^1_0(\Omega_{\eps_n})$ uniformly converges to an eigenfunctions $u_{k-1}\in H^1_0(\Omega^\ast)$, corresponding to $\lambda_k(\Omega^\ast)=\lambda_{k-1}(\Omega^\ast)$. Since the Lipschitz constants of $u_{k-1}^{n}$ are uniform, we have the conclusion.

If the case $(b)$ occurs, then reasoning as in the case $(ii)$, we have that $\Omega_{\eps_1}=\Omega^\ast$. Indeed, we have
\begin{equation}\label{thk5}
\begin{array}{ll}
(1-\eps_0)\lambda_k(\Omega_{\eps_1})+\eps_0\lambda_{k-1}(\Omega_{\eps_1})+3\Lambda|\Omega_{\eps_1}|\\
\qquad=(1-\eps_0)\lambda_k(\Omega_{\eps_1})+\eps_0\big[(1-\eps_1)\lambda_{k-1}(\Omega_{\eps_1})+\eps_1\lambda_{k-2}(\Omega_{\eps_1})\big]+3\Lambda|\Omega_{\eps_1}|\\
\qquad\le(1-\eps_0)\lambda_k(\Omega^\ast)+\eps_0\big[(1-\eps_1)\lambda_{k-1}(\Omega^\ast)+\eps_1\lambda_{k-2}(\Omega^\ast)\big]+3\Lambda|\Omega^\ast|\\
\qquad=(1-\eps_0)\lambda_k(\Omega^\ast)+\eps_0\lambda_{k-1}(\Omega^\ast)+3\Lambda|\Omega^\ast|.
\end{array}
\end{equation}
On the other hand, we supposed that $\Omega^\ast$ is a solution of \eqref{thk1} with $\eps=\eps_0$ and so, it is the unique minimizer of the problem
\begin{equation}\label{thk6}
\min\Big\{(1-\eps_0)\lambda_k(\Omega)+\eps_0\lambda_{k-1}(\Omega)+3\Lambda|\Omega|:\ \Omega\supset\Omega^\ast\Big\}.
\end{equation} 
Thus, we have $\Omega^\ast=\Omega_{\eps_1}$. We proceed considering, for any $\eps\in(0,1)$, the problem  
\begin{equation}\label{thk7}
\begin{array}{ll}
\min\Big\{(1-\eps_0)\lambda_k(\Omega)+\eps_0(1-\eps_1)\lambda_{k-1}(\Omega)\\
\qquad\qquad+\eps_0\eps_1\big[(1-\eps)\lambda_{k-2}(\Omega)+\eps\lambda_{k-3}(\Omega)\big]+4\Lambda|\Omega|:\ \Omega\supset\Omega^\ast\Big\},
\end{array}
\end{equation} 
and repeat the procedure described above. We note that this procedure stops after at most $k$ iterations. Indeed, if $\Omega^\ast$ is a shape quasi-minimizer for $\lambda_1$ and $\lambda_k(\Omega^\ast)=\dots=\lambda_1(\Omega^\ast)$, then we obtain the result applying Lemma \ref{double3} for $k=1$.
\end{enumerate}
\end{proof}

As a consequence, we obtain the following result for the optimal set for the $k^{th}$ Dirichlet eigenvalue. 
\begin{cor}
Let $\Omega$ be a solution of the problem
$$\min\Big\{\lambda_k(\Omega):\ \Omega\subset\R^d,\ \Omega\ \hbox{quasi-open},\ |\Omega|=1\Big\}.$$
Then there exists an eigenfunction $u_k\in H^1_0(\Omega)$, corresponding to the eigenvalue $\lambda_k(\Omega)$, which is Lipschitz continuous on $\R^d$.
\end{cor}

%

\begin{oss}\label{thf}
We note that Theorem \ref{thk} can be used to obtain information for the supersolutions of general spectral functionals. Let $\mathcal{F}:\mathcal{A}\to\R$ be a functional defined on the family of sets of finite measure $\mathcal{A}$ and suppose that there exist non-negative real numbers $c_k,\ k\in\N$, such that for each couple of sets $\Omega\subset\widetilde\Omega\subset\R^d$ of finite measure we have
$$c_k\big(\lambda_k(\Omega)-\lambda_{k}(\widetilde\Omega)\big)\le \mathcal{F}(\Omega)-\mathcal{F}(\widetilde\Omega).$$
If $\Omega$ is a shape supersolution for $\mathcal{F}+\Lambda|\cdot|$, then for any $k\in\N$ such that $c_k>0$, there is an eigenfunction $u_k\in H^1_0(\Omega)$, normalized in $L^2$ and corresponding to $\lambda_k(\Omega)$, which is Lipschitz continuous on $\R^d$. Indeed, it is enough to note that, whenever $c_k>0$, we have 
$$\lambda_k(\Omega)-\lambda_k(\widetilde\Omega)\le c_k^{-1}\left(\mathcal{F}(\Omega)-\mathcal{F}(\widetilde\Omega)\right)\le c_k^{-1}\Lambda|\widetilde\Omega\setminus\Omega|.$$
The conclusion follows by Theorem \ref{thk}. 
\end{oss}

In order to prove a regularity result which involves all the eigenfunction corresponding to the eigenvalues that appear in functionals of the form $F\big(\lb_{k_1}(\Omega), \dots, \lb_{k_p}(\Omega)\big)$, we need the following preliminary result.

\begin{lemma}\label{double3treno2}
Let $\Omega^\ast\subset\R^d$ be a shape supersolution for the functional $$\Omega\mapsto \lambda_k(\Omega)+\lambda_{k+1}(\Omega)+\dots+\lambda_{k+p}(\Omega)+\Lambda|\Omega|,$$
for some constant $\Lambda>0$. Then there are $L^2$-orthonormal  eigenfunctions $u_k,\dots, u_{k+p}\in\widetilde H^1_0(\Omega^\ast)$, corresponding to the eigenvalues $\lambda_k(\Omega^\ast),\dots,\lambda_{k+p}(\Omega^\ast)$, which are Lipschitz continuous on $\R^d$. 
\end{lemma}
\begin{proof}
We prove the lemma in two steps.

\emph{Step 1.} Suppose that $\lambda_{k}(\Omega^\ast)>\lambda_{k-1}(\Omega^\ast)$. We first note that, by Lemma \ref{double3}, if $j\in\{k,k+1,\dots,k+p\}$ is such that $\lambda_j(\Omega^\ast)>\lambda_{j-1}(\Omega^\ast)$, then any eigenfunction, corresponding to the eigenvalue $\lambda_j(\Omega^\ast)$, is Lipschitz continuous on $\R^d$. Let us now divide the eigenvalues $\lambda_k(\Omega^\ast),\dots,\lambda_{k+p}(\Omega^\ast)$ into clusters of equal consecutive eigenvalues. There exists    $k= k_1< k_2< \dots< k_s\le k+p$ such that
\begin{align*}
\lambda_{k-1}(\Omega^\ast)&<\lambda_{k_1}(\Omega^\ast)=\dots=\lambda_{k_2-1}(\Omega^\ast)\\
&<\lambda_{k_2}(\Omega^\ast)=\dots=\lambda_{k_3-1}(\Omega^\ast)\\
&\dots\\
&<\lambda_{k_s}(\Omega^\ast)=\dots=\lambda_{k+p}(\Omega^\ast).
\end{align*}
 Then, by the above observation, the eigenspaces corresponding to the eigenvalues
$$\lambda_{k_1}(\Omega^\ast),\lambda_{k_2}(\Omega^\ast),\dots,\lambda_{k+p}(\Omega^\ast),$$
consist on Lipschitz continuous functions. In particular, there exists a  sequence of consecutive eigenfunctions $u_k,\dots,u_{k+p}$ satisfying the claim of the lemma.  

\emph{Step 2.}  Suppose now that $\lambda_{k}(\Omega^\ast)=\lambda_{k-1}(\Omega^\ast)$. For each $\eps\in(0,1)$ we consider the problem 
\begin{equation}\label{thf.1e3}
\min\Big\{\sum_{j=1}^p\lambda_{k+j}(\Omega)+(1-\eps)\lambda_k(\Omega)+\eps\lambda_{k-1}(\Omega)+2\Lambda|\Omega|:\ \Omega^\ast\subset\Omega\subset\R^d\Big\}.
\end{equation}
As in Theorem \ref{thk}, we have that at least one of the following cases occur:
\begin{enumerate}[(i)]
\item There is a sequence $\eps_n\to0$ and a corresponding sequence $\Omega_{\eps_n}$ of minimizers of \eqref{thf.1e3} such that, for each $n\in\N$, 
$$\lambda_{k}(\Omega_{\eps_n})>\lambda_{k-1}(\Omega_{\eps_n}).$$

\item There is some $\eps_0\in(0,1)$ for which there is $\Omega_{\eps_0}$ a solution of \eqref{thf.1e3} such that
$$\lambda_{k}(\Omega_{\eps_0})=\lambda_{k-1}(\Omega_{\eps_0}).$$
\end{enumerate}
   In the first case $\Omega_{\eps_n}$ is a shape supersolution for the functional $$\Omega\mapsto\lambda_k(\Omega)+\dots+\lambda_{k+p}(\Omega)+(1-\eps_n)^{-1}\Lambda|\Omega|.$$
Thus, by \emph{Step 1}, there are orthonormal eigenfunctions $u_k^n,\dots,u_{k+p}^n\in H^1_0(\Omega_{\eps_n})$, which are Lipschitz continuous on $\R^d$. Using the same approximation argument from Theorem \ref{thk}, we obtain the claim. 
   In the second case, reasoning again as in Theorem \ref{thk}, we have that $\Omega_{\eps_0}=\Omega^\ast$ and we have to consider two more cases. If $\lambda_{k-1}(\Omega^\ast)>\lambda_{k-2}(\Omega^\ast)$, we have the claim by \emph{Step 1}. If $\lambda_{k-1}(\Omega^\ast)=\lambda_{k-2}(\Omega^\ast)$, then we consider the problem 
\begin{equation*}\label{thf.1e3.bis}
\min\Big\{\sum_{j=1}^p\lambda_{k+j}(\Omega)+(1-\eps_0)\lambda_k(\Omega)+\eps_0\big[(1-\eps)\lambda_{k-1}(\Omega)+\eps\lambda_{k-2}(\Omega)\big]+3\Lambda|\Omega|:\ \Omega^\ast\subset\Omega\subset\R^d\Big\},
\end{equation*}
and proceed by repeating the argument above, until we obtain the claim or until we have a functional involving $\lambda_1$, in which case we apply one more time the result from \emph{Step 1}.
\end{proof}

\medskip

Before we state our main result (Theorem \ref{thf.1}), we recall that:
\begin{itemize}
\item for two points $\ x:=(x_1,\dots,x_p)\in\R^p\ $ and $\ y:=(y_1,\dots,y_p)\in\R^p\ $,
we say that $x\ge y\ $ if and only if $\ x_i\ge y_i,\ $ for all $i=1,\dots,p$.
\item we say that a functions $F:\R^p\to\R$ is \emph{bi-Lipschitz in each variable}, if $F$ is Lipschitz and there are positive real constants $c_1,\dots,c_p\in(0,+\infty)$ such that 
\begin{equation}\label{bi.lip}
F(x)-F(y)\ge c_1(x_1-y_1)+\dots+c_p(x_p-y_p),\qquad \forall x,y\in\R^p\ \hbox{ s.t. }\ x\ge y.
\end{equation}
\item we say that we say $F:\R^p\to\R$ is \emph{locally bi-Lipschitz in each variable}, if the inequality \eqref{bi.lip} holds for each $y$ in a neighbourhood of $x$.  
\end{itemize}

\begin{teo}\label{thf.1}
Let $F: \R^p\rightarrow \R$ be an increasing and locally bi-Lipschitz function
in each variable and let $0< k_1< k_2<\dots < k_p$ be natural numbers. 
Then for every bounded shape supersolution $\Omega^\ast$ of the functional 
$$\Om \mapsto F\big(\lb_{k_1}(\Omega), \dots, \lb_{k_p}(\Omega)\big)+\Lambda|\Omega|,$$ 
there exists a sequence of orthonormal eigenfunctions $u_{k_1},\dots, u_{k_p}$, corresponding to the eigenvalues $ \lb_{k_j}(\Omega^\ast)$, $j=1,\dots, p$, which are Lipschitz continuous on $\R^d$. Moreover,
\begin{itemize}
\item
if for some $k_j$ we have $\lb_{k_j}(\Om^\ast)> \lb_{k_j-1}(\Om^\ast)$, then the full eigenspace corresponding to $\lb_{k_j}(\Om^\ast)$ consists only on Lipschitz functions;

\item
if $\lb_{k_j}(\Om^\ast)=\lb_{k_{j-1}}(\Om^\ast)$, then there exist at least $k_j-k_{j-1}+1$ orthonormal Lipschitz eigenfunctions corresponding to $\lb_{k_j}(\Om^\ast)$.
\end{itemize}
\end{teo}
\begin{proof}
Let $c_1,\dots,c_p\in\R^+$ be the strictly positive real numbers from \eqref{bi.lip} 
We note that if $\Omega^\ast$ is a supersolution of $F(\lb_{k_1}, \dots, \lb_{k_p})$, then $\Omega^\ast$ is also a supersolution for the functional 
$$\widetilde F=\left[\min_{j\in\{1,\dots,p\}}c_j\right]\left(\lambda_{k_1}+\dots+\lambda_{k_p}\right),$$
and, since $\min_{j\in\{1,\dots,p\}}c_j>0$, we can assume $\min_{j\in\{1,\dots,p\}}c_j=1$.  

Reasoning as in Lemma \ref{double3treno2}, we divide the family  $ \big(\lambda_{k_1}(\Omega^\ast),\dots,\lambda_{k_p}(\Omega^\ast) \big)$ into clusters of equal eigenvalues with consecutive indexes. There exist $1\le i_1< i_2\dots< i_s\le p-1$ such that
\begin{align*}
\lambda_{k_1}(\Omega^\ast)=\dots=\lambda_{k_{i_1}}(\Omega^\ast)&<\lambda_{k_{(i_1+1)}}(\Omega^\ast)=\dots=\lambda_{k_{i_2}}(\Omega^\ast)\\
&<\lambda_{k_{(i_2+1)}}(\Omega^\ast)=\dots=\lambda_{k_{i_3}}(\Omega^\ast)\\
&\dots\\
&<\lambda_{k_{(i_s+1)}}(\Omega^\ast)=\dots=\lambda_{k_p}(\Omega^\ast).
\end{align*}
 Since the eigenspaces, corresponding to different clusters, are orthogonal to each other, it is enough to prove the claim for the functionals defined as the sum of the eigenvalues in each cluster. In  other words, it is sufficient to restrict our attention only to the case when $\Omega^\ast$ is a supersolution for the functional $F(\lambda_{k_1},\dots,\lambda_{k_p})+\Lambda|\cdot|:=\sum_{j=1}^p\lambda_{k_j}+\Lambda|\cdot|$ and is such that 
\begin{equation}\label{thf.1e1}
\lambda_{k_1}(\Omega^\ast)=\dots=\lambda_{k_p}(\Omega^\ast).
\end{equation}
Moreover, in this case $\Omega^\ast$ is also a shape supersolution (with possibly different constant $\Lambda$) for the sum of consecutive eigenvalues $\sum_{k=k_1}^{k_p}\lambda_k+\Lambda|\cdot|$. Indeed, 
it is enough to consider the functional
$$\widetilde F(\Om)=\frac12\sum_{j=1}^p\lambda_{k_j}(\Om)+\theta\sum_{k=k_1}^{k_p}\lambda_{k}(\Om),$$
for a suitable value of $\theta$ (e.g. $\theta= \frac{1}{2(k_p-k_1+1)}$). The conclusion then follows by Lemma \ref{double3treno2}.
   
\end{proof}

\section{Optimal sets for functionals depending on the first $k$ eigenvalues}\label{openness}
In this last Section we aim to show that, at least for some specific functionals, we can conclude that a minimizer is actually equivalent to an open set.
All the following results are, essentially, consequences of Theorem~\ref{thf.1}.

\begin{teo}\label{thf.100}
Let $F:\R^k\to\R$ be an increasing function locally bi-Lipschitz in each variable. Then every solution $\Omega^\ast$ of the problem 
\begin{equation}\label{sopmilan}
\min\Big\{F\big(\lambda_1(\Omega),\dots,\lambda_k(\Omega)\big):\ \Omega\subset\R^d\ \hbox{measurable},\ |\Omega|=1\Big\},
\end{equation}
is an open set. Moreover, the eigenfunctions of the Dirichlet Laplacian on $\Omega^\ast$, corresponding to the eigenvalues $\lambda_1(\Omega^\ast),\dots,\lambda_k(\Omega^\ast)$, are Lipschitz continuous on $\R^d$.
\end{teo}
\begin{proof}
We first note that the existence of a solution of \eqref{sopmilan} follows by the results from~\cite{bulbk} and~\cite{mp}. Then, we prove that every solution $\Omega^\ast$ is a local shape supersolution of the functional 
$$\Omega\mapsto F\big(\lambda_1(\Omega),\dots,\lambda_k(\Omega)\big)+\Lambda|\Omega|,$$
for some suitably chosen $\Lambda>0$. Indeed, let $\Omega^\ast\subset\Omega$ and let $\ds t:=\left(\frac{|\Omega|}{|\Omega^\ast|}\right)^{1/d}>1$. By the optimality of $\Omega^\ast$, we have
\begin{align*}
F\big(\lambda_1(\Omega^\ast),\dots,\lambda_k(\Omega^\ast)\big)&\le F\big(\lambda_1(\Omega/t),\dots,\lambda_k(\Omega/t)\big)\\
&\le F\big(\lambda_1(\Omega),\dots,\lambda_k(\Omega)\big)\\
&\qquad+\Big(F\big(t^2\lambda_1(\Omega),\dots,t^2\lambda_k(\Omega)\big)-F\big(\lambda_1(\Omega),\dots,\lambda_k(\Omega)\big)\Big)\\
&\le F\big(\lambda_1(\Omega),\dots,\lambda_k(\Omega)\big)+\hbox{Lip}(F)(t^2-1)\sum_{i=1}^k\lambda_i(\Omega)\\
&\le F\big(\lambda_1(\Omega),\dots,\lambda_k(\Omega)\big)+\hbox{Lip}(F)(t^d-1)\sum_{i=1}^k\lambda_i(\Omega^\ast)\\
&\le F\big(\lambda_1(\Omega),\dots,\lambda_k(\Omega)\big)+\hbox{Lip}(F)\left(\sum_{i=1}^k\lambda_i(\Omega^\ast)\right)|\Omega^\ast|^{-1}\Big(|\Omega|-|\Omega^\ast|\Big),
\end{align*}
where $\hbox{Lip}(F)$ is the Lipschitz constant of $F$ and we finally set $\ds\Lambda:=\frac{\hbox{Lip}(F)}{|\Omega^\ast|}\left(\sum_{i=1}^k\lambda_i(\Omega^\ast)\right).$
Now the Lipschitz continuity of the eigenfunctions $u_1,\dots,u_k$ on $\Omega^\ast$ follows by Theorem \ref{thf.1}. The openness of the set $\Omega^\ast$ follows by the observation that the set 
$$\Omega^{\ast\ast}:=\bigcup_{i=1}^k\{u_k\neq0\}\subset\Omega^\ast,$$
is open and has the same eigenvalues, up to order $k$, as $\Omega^\ast$. By the optimality of $\Omega^\ast$ we have $|\Omega^\ast\Delta \Omega^{\ast\ast}|=0$.
\end{proof}

\begin{oss}
The openness of the optimal set from Theorem \ref{thf.1} can also be obtained reasoning on each connected component of $\Omega^\ast$ and applying the Alt-Caffarelli technique from \cite{altcaf} for the functional $\lambda_1(\Omega)+\Lambda|\Omega|$ as in \cite{brla} and \cite{bubuve}.  
\end{oss}

\begin{oss}
In two dimensions, it is possible to obtain the continuity of the eigenfunctions from Theorem \ref{thf.100} by a more direct method involving only elementary tools (see \cite{phdm}). Roughly speaking, using the argument from Remark \ref{rmkcont}, one can prove that in each level set of some of the eigenfunctions, there cannot be holes of small diameter, since otherwise it is more convenient to ``fill" them. More precisely, for every $\xi>0$ and every $x\in\R^2$ such that $u_1^2(x)+\dots+u_k^2(x)>\xi$ there is a constant $r=r(\xi)>0$ and a ball, of radius $r(\xi)$ and centred in $x$, which is entirely contained in $\Omega^\ast$. In particular, this fact provides an estimate on the modulus of continuity of the function $U:=u_1^2+\dots+u_k^2$ on the boundary of $\Omega^\ast$.
\end{oss}

By the definition of the open set $\Omega^{\ast\ast}$, we have that the first $k$ elements of the spectrum of the Dirichlet Laplacian, defined on the space $\widetilde H^1_0(\Omega^{\ast\ast})$, and those, defined on the classical Sobolev space $H^1_0(\Omega^{\ast\ast})$, coincide. Thus, we have a solution of the shape optimization problem \eqref{sopmilan} in its classical formulation.   

\begin{cor}\label{corf.100}
Let $F:\R^k\to\R$ be an increasing function locally bi-Lipschitz in each variable. Then there is a solution $\Omega^\ast$ of the problem 
\begin{equation}\label{sopmil}
\min\Big\{F\big(\lambda_1(\Omega),\dots,\lambda_k(\Omega)\big):\ \Omega\subset\R^d\ \hbox{open},\ |\Omega|=1\Big\}.
\end{equation}
Moreover, the eigenfunctions of the Dirichlet Laplacian on $\Omega^\ast$, corresponding to the eigenvalues $\lambda_1(\Omega^\ast),\dots,\lambda_k(\Omega^\ast)$, are Lipschitz continuous on $\R^d$.
\end{cor}

\begin{oss}
Theorem \ref{thf.100} and Corollary \ref{corf.100} apply, in particular, to the functional 
$$F\big(\lambda_1(\Omega),\dots,\lambda_k(\Omega)\big):=\sum_{i=1}^k\lambda_i(\Omega).$$
\end{oss}

In Theorem \ref{thf.100} we proved that every solution of \eqref{sopmilan} contains another solution, which is an open set. The analogous result holds also for supersolutions. 

\begin{prop}
Let $F:\R^k\to\R$ be an increasing locally bi-Lipschitz function in each variable and $\Omega^\ast$ be a subsolution for the functional
\begin{equation}\label{funsum}
\Omega\mapsto F\big(\lambda_1(\Omega),\dots,\lambda_k(\Omega)\big)+\Lambda|\Omega|.
\end{equation}
Then
\begin{enumerate}[(i)]
\item There are eigenfunctions $u_1,\dots,u_k\in \widetilde H^1_0(\Omega^\ast)$, corresponding to the eigenvalues\\ $\lambda_1(\Omega^\ast),\dots,\lambda_k(\Omega^\ast)$, which are Lipschitz continuous on $\R^d$. 
\item There is an open set $\Omega^{\ast\ast}\subset\Omega^\ast$ such that $u_1,\dots,u_k\in H^1_0(\Omega^{\ast\ast})$; $\lambda_i(\Omega^{\ast\ast})=\lambda_i(\Omega^\ast)$, for every $i=1,\dots,k$; $\Omega^{\ast\ast}$ is still a supersolution for the functional \eqref{funsum}. 
\end{enumerate}
\end{prop}
\begin{proof}
%
%
%
The first claim follows from Theorem \ref{thf.1}. For {\it (ii)} we define $\Omega^{\ast\ast}$ as in Theorem \ref{thf.100}:
\[
\Om^{**}:=\bigcup_{i=1}^k{\left\{u_i\neq0\right\}},
\]
thus we have $\lambda_i(\Om^*)=\lambda_i(\Om^{**})$ for every $i=1,\dots, k$.
For all $\Om\supset\Om^{**}$ we compute
\[
\begin{array}{ll}
\ds F\big(\lambda_1(\Omega^{\ast\ast}),\dots,\lambda_k(\Omega^{\ast\ast})\big)+&\ds\Lambda|\Omega^{**}|+\Lambda|\Omega^*\setminus\Omega^{**}|= F\big(\lambda_1(\Omega^\ast),\dots,\lambda_k(\Omega^\ast)\big)+\Lambda|\Omega^*|\\
\\
&\ds\le F\big(\lambda_1(\Omega\cup\Omega^\ast),\dots,\lambda_k(\Omega\cup\Omega^\ast)\big)+\Lambda|\Omega\cup\Omega^*|\\
\\
&\ds\le F\big(\lambda_1(\Omega),\dots,\lambda_k(\Omega)\big)+\Lambda|\Omega|+\Lambda|\Omega^*\setminus\Omega^{\ast\ast}|,
\end{array}
\]
hence $\Om^{**}$ is also a supersolution for \eqref{funsum}. 
\end{proof}


For functionals of the form 
$$\Omega\mapsto F\big(\lambda_{k_1}(\Omega),\dots,\lambda_{k_p}(\Omega)\big),$$
depending on some non-consecutive eigenvalues $\lambda_{k_1},\dots,\lambda_{k_p}$, it is still possible to obtain that an optimal set $\Omega^\ast$  for the problem 
\begin{equation}\label{bilipincr}
\min\Big\{F\big(\lambda_{k_1}(\Omega),\dots,\lambda_{k_p}(\Omega)\big)\;:\;\Omega\subset\R^d\;\mbox{measurable, }|\Omega|=1\Big\},
\end{equation} 
is open, provided that an additional condition on the eigenvalues of $\Omega^\ast$ is satisfied.
\begin{prop}\label{openmeasure}
Let $F:\R^p\ra \R$ be an increasing and locally bi-Lipschitz function in each variable, $0< k_1< k_2<\dots < k_p$ be natural numbers and $\Om^*$ a minimizer for the problem \eqref{bilipincr}.
If for all $j=1,\dots,p$ we have $\la_{k_j}(\Om^*)>\la_{k_j-1}(\Om^*)$ then the set $\Omega^\ast$ is open. Moreover all the eigenfunctions corresponding to $\lambda_{k_j}(\Omega^\ast)$, for all $j=1,\dots, p$ are Lipschitz continuous on $\R^d$.
\end{prop}
\begin{proof}
The second part of the claim follows by Theorem \ref{thf.1}. In order to prove the openness of $\Omega^\ast$ we consider the family of indices 
$$I:=\Big\{i\in\N\;:\;\lambda_i(\Omega^*)=\lambda_{k_j}(\Omega^*),\;\mbox{for some }j\Big\},$$
and the set 
\[
\Om_A:=\left\{x\in\R^d\;:\;\sum_{i\in I}{u_i(x)^2}>0\right\}.
\]
We aim to prove that the set $N:=\Om^*\setminus \Om_A$ has zero Lebesgue measure. Suppose, by contradiction, that $|N|>0$ and let $x\in N$ be a point of density one for $N$, i.e. 
\begin{equation}
\lim_{\rho\rightarrow 0}{\frac{|N\cap B_\rho(x)|}{|B_{\rho}(x)|}}=1.
\end{equation}
Since, for $\rho\to0$, the sets $\Omega^*\setminus (N\cap B_\rho(x))$ $\gamma$-converge to $\Omega^*$ we have the convergence of the spectra $\lambda_k(\Omega^*\setminus (N\cap B_\rho(x)))\to\lambda_k(\Omega^\ast)$, for every $k\in\N$.

Since $\la_{k_j}(\Omega^*)>\la_{k_j-1}(\Omega^*)$, we can choose $\rho$ small enough such that the new set $\widetilde \Omega:=\Omega^*\setminus (N\cap B_{\rho}(x))$ satisfies 
\begin{equation}\label{simpleineq}
\lambda_i(\widetilde \Omega)<\lambda_{k_j}(\Omega^\ast),\qquad \forall\;i=1,\dots,k_j-1.
\end{equation}
We now note that for $i\in I$ the eigenfunction $u_i\in \widetilde H^1_0(\widetilde\Omega)$ and since $\widetilde\Omega\subset\Omega^\ast$, we get that $u_i$ satisfies the equation
\[
-\Delta u_i=\lambda_{k_j}(\Omega^*)u_i,\qquad u_i\in \widetilde H^1_0(\widetilde \Omega).
\]
Thus, for $i\in I$, the number $\lambda_i(\Omega^\ast)$ is also in the spectrum of the Dirichlet Laplacian on $\widetilde\Omega$. Combined with \eqref{simpleineq} this gives 
\begin{equation}\label{lipeig}
\lambda_k(\widetilde \Omega)\le \lambda_{k}(\Omega^\ast),\qquad \forall\  k=1,\dots,k_p.
\end{equation}
Since for $\rho>0$ small enough $|N\cap B_\rho(x)|>0$, we have that $|\widetilde\Omega|<|\Omega^\ast|$. By the strict monotonicity of $F$, we can rescale $\widetilde\Omega$ thus obtaining a better competitor than $\Omega^\ast$ in \eqref{bilipincr}, which is a contradiction with the optimality of $\Omega^\ast$. 
%
%
\end{proof}

\begin{oss}
Unfortunately, Proposition~\ref{openmeasure} provides the openness of optimal sets only up to zero Lebesgue measure. Hence we have that $\widetilde H^1_0(\Om^*)=\widetilde H^1_0(\Om_A)$, but we do not know in general if $H^1_0(\Om^*)=H^1_0(\Om_A)$.
\end{oss}

\appendix
\section{Appendix: Proof of Theorem \ref{main}}
For the sake of the completeness, we report here the proof of Theorem \ref{main}, given in \cite{bhp05}. We note that if the state function $u$, quasi-minimizer for the functional $J_f$, is positive, then the classical approach of Alt and Caffarelli (see \cite{altcaf}) can be applied to obtain the Lipschitz continuity of $u$. This approach is based on an external perturbation and on the following inequality (see \cite[Lemma 3.2]{altcaf})
\begin{equation}\label{altcafineq}
\frac{\big|B_r(x_0)\cap\{u=0\}\big|}{r^2}\Big(\mean{\partial B_r(x_0)}{u\,d\H^{d-1}}\Big)^2\le C_d\int_{B_r(x_0)}|\nabla(u-v)|^2\,dx,
\end{equation}
which holds for every $x_0\in\R^d$, $r>0$, $u\in H^1(\R^d)$, $u\ge 0$ and $v\in H^1(B_r)$ that solves
\begin{equation}\label{altcafv}
\min\Big\{\int_{B_r(x_0)}|\nabla v|^2\,dx:\ v-u\in H^1_0(B_r(x_0)),\ v\ge u \Big\}.
\end{equation}
Since for sign-changing state functions $u$, the inequality \eqref{altcafineq} is not known, one needs a more careful analysis on the common boundary of $\{u>0\}$ and $\{u<0\}$, which is based on the monotonicity formula of Alt-Caffarelli-Friedmann.
\begin{teo}\label{mth}
Let $U^+,U^-\in H^1(B_1)$ be continuous non-negative functions such that $\Delta U^\pm\ge-1$ on $B_1$ and $U^+U^-=0$. Then there is a dimensional constant $C_d$ such that for each $r\in(0,\frac12)$
\begin{equation}\label{mthe1}
\left(\frac{1}{r^2}\int_{B_r}\frac{|\nabla U^+(x)|^2}{|x|^{d-2}}\,dx\right)\left(\frac{1}{r^2}\int_{B_r}\frac{|\nabla U^-(x)|^2}{|x|^{d-2}}\,dx\right)\le C_d\left(1+\int_{B_1}|U^++U^-|^2\,dx\right).
\end{equation}   
\end{teo}

For our purposes we will need the following rescaled version of this formula.

\begin{cor}\label{mon}
Let $\Omega\subset\R^d$ be a quasi-open set of finite measure, $f\in L^\infty(\Omega)$ and $u:\R^d\to\R$ be a continuous function such that
\begin{equation}\label{mone1}
-\Delta u=f\quad \hbox{in}\quad \Omega,\qquad u\in H^1_0(\Omega).
\end{equation}
Setting $u^+=\sup\{u,0\}$ and $u^-=\sup\{-u,0\}$, there is a dimensional constant $C_d$ such that for each $0<r\le1/2$
\begin{equation}\label{mone2}
\left(\frac{1}{r^2}\int_{B_r}\frac{|\nabla u^+(x)|^2}{|x|^{d-2}}\,dx\right)\left(\frac{1}{r^2}\int_{B_r}\frac{|\nabla u^-(x)|^2}{|x|^{d-2}}\,dx\right)\le C_d\left(\|f\|_{\infty}^2+\int_\Omega u^2\,dx\right)\le C_m,
\end{equation} 
where $C_m=C_d\|f\|_\infty^2\left(1+|\Omega|^{\frac{d+4}{d}}\right)$.
\end{cor}
\begin{proof}
We apply Theorem \ref{mth} to $U^\pm=\|f\|_\infty^{-1}u^\pm$ and substituting in \eqref{mthe1} we obtain the first inequality  in \eqref{mone2}. The second one follows, using the equation \eqref{mone1}:
\begin{equation}\label{mone3}
\|u\|_{L^2}^2\le C_d|\Omega|^{2/d}\|\nabla u\|_{L^2}^2=C_d|\Omega|^{2/d}\int_\Omega fu\,dx\le C_d|\Omega|^{2/d+1/2}\|f\|_{\infty}\|u\|_{L^2}.
\end{equation} 
\end{proof}

The proof of the Lipschitz continuity of the quasi-minimizers for $J_f$ needs two preliminary results, precisely in Lemma \ref{contstate} we prove the continuity of $u$ and in Lemma \ref{l2}, we give an estimate on the Laplacian of $u$ as a measure on the boundary $\partial\{u\neq0\}$. 

\begin{lemma}\label{contstate}
Suppose that $u$ satisfies the conditions $(a)$ and $(b)$ from Theorem \ref{main}. Then $u$ is continuous. 
\end{lemma}
\begin{proof}
Let $x_n\to x_\infty\in\R^d$ and set $\delta_n:=|x_n-x_\infty|$. If for some $n$, $|B(x_\infty,\delta_n)\cap\{u=0\}|=0$, then $-\Delta u=f$ in $B(x_\infty,\delta_n)$ and so $u$ is continuous in $x_\infty$.\\
Assume now that for all $n$, $|B(x_\infty,\delta_n)\cap\{u=0\}|\neq0$ and consider the function $u_n:\R^d\to\R$ defined by $u_n(\xi)=u(x_\infty+\delta_n\xi)$. Since $\|u_n\|_\infty=\|u\|_\infty$, for any $n$, we can assume, up to a subsequence, that $u_n$ converges weakly-$\ast$ in $L^\infty$ to some function $u_\infty\in L^\infty(\R^d)$.\\
\emph{If we prove that $u_\infty=0$ and that $u_n\to u_\infty$ uniformly on $B_1$, then we would have that $u$ is continuous and $u(x_\infty)=0$.}\\
\emph{Step 1. $u_\infty$ is a constant.}\\
For all $R\ge1$ and $n\in\N$, we introduce the function $v_{R,n}$ such thay:
\begin{equation}
\begin{cases}
\begin{array}{rl}
-\Delta v_{R,n}=f,&\ \hbox{in}\ B_{R\delta_n}(x_\infty),\\
v_{R,n}=u,&\ \hbox{on}\ \partial B_{R\delta_n}(x_\infty).
\end{array}
\end{cases}
\end{equation}

Setting $v_n(\xi):=v_{R,n}(x_\infty+\delta_n\xi)$, we have that 
\begin{align*}
\int_{B_R}|\nabla(u_n-v_n)|^2d\xi&=\delta_n^{2-d}\int_{B(x_\infty,R\delta_n)}|\nabla(u-v_{R,n})|^2dx\\
&=\delta_n^{2-d}\int_{B(x_\infty,R\delta_n)}\nabla u\cdot\nabla(u-v_{R,n})dx-\delta_n^{2-d}\int_{B(x_\infty,R\delta_n)} f(u-v_{R,n})dx\\
&\le C_b\delta_n^{2-d}\left(\int_{B(x_\infty,R\delta_n)}|\nabla(u-v_{R,n})|^2dx\right)^{1/2}R^{d/2}\delta_n^{d/2}\\
&\le C_bR^{d/2}\delta_n\left(\int_{B_R}|\nabla(u_n-v_{n})|^2d\xi\right)^{1/2},
\end{align*}
and thus, for $\delta_n\le r_0$, we have 
\begin{equation}
\mean{B_R}{|\nabla(u_n-v_n)|^2\,d\xi}\le C_b^2\delta_n^2,
\end{equation}
where $C_b$ is the constant from \eqref{b}. In particular, $u_n-v_n\to 0$ in $H^1(B_R)$ for any $R\ge1$. On the other hand, we have that
\begin{equation}
\begin{cases}
\begin{array}{rl}
-\Delta v_n=\delta_n^2f,&\ \hbox{in}\ B_R,\\
v_n\le\|u\|_\infty,&\ \hbox{on}\ \partial B_R.
\end{array}
\end{cases}
\end{equation}
Thus, $v_n$ are equi-bounded (by the maximum principle) and equi-continuous (by Remark \ref{gradest}) on the ball $B_{R/2}$ and so, the sequence $v_n$ uniformly converges to some function which is harmonic on $B_{R/2}$. By the uniqueness of the weak-$\ast$ limit in $L^\infty$, we have that this function is precisely $L^\infty$. Thus, $u_\infty$ is a harmonic function on each $B_{R/2}$ and so, on $\R^d$. Since it is bounded, it is a constant.\\ 
\emph{Step 2. $u_n\to u_\infty$ in $H^1_{loc}(\R^d)$.}\\
 In fact, for the functions $\widetilde v_n=v_n-u_\infty$, we have that 
\begin{equation}
\begin{cases}
\begin{array}{rl}
-\Delta\widetilde v_n=\delta_n^2f,&\ \hbox{in}\ B_R,\\
\widetilde v_n\le2\|u\|_\infty,&\ \hbox{on}\ \partial B_R,
\end{array}
\end{cases}
\end{equation}
and $\widetilde v_n\to0$ uniformly on $B_{R/2}$. By Remark \ref{gradest}, we have that $\|\nabla\widetilde v_n\|_{L^\infty(B_{R/4})}\to0$ and so, $v_n\to u_\infty$ in $H^1(B_{R/4})$ and the same holds for $u_n$.\\

\emph{Step 3. If $u_\infty\ge 0$, then $u_n^-\to0$ uniformly on balls.}\\
Since on $\{u_n<0\}$, the equality $-\Delta u_n^-=-\delta_n^2f$ holds, we have that $-\Delta u_n^-\le -\delta_n^2fI_{\{u_n<0\}}\le\delta_n^2|f|$ on $\R^d$. Thus, it is enough to prove that for each $R\ge1$, $\widetilde u_n\to 0$ uniformly on $B_{R/2}$, where 
\begin{equation}
\begin{cases}
\begin{array}{rl}
-\Delta \widetilde u_n=\delta_n^2|f|, &\ \hbox{in}\ B_R,\\
\widetilde u_n=u_n^-,&\ \hbox{on}\  \partial B_R.  
\end{array}
\end{cases}
\end{equation}
Since $u_n^-\to0$ in $H^1(B_R)$, we have that $\int_{\partial B_R}u_n^-\to 0$. Writing $\widetilde u_n=\widetilde w_n+\widetilde u_h$, where $\widetilde w_n\in H^1_0(B_R)$, $-\Delta \widetilde w_n=\delta_n^2|f|$ and $\widetilde u_h$ is the harmonic function on $B_R$ with boundary values equal to $\widetilde u_n$, we have the thesis of Step 3.\\

\emph{Step 4. $u_\infty=0$}\\
Suppose that $u_\infty\ge0$. Let $y_n=x_\infty+\delta_n\xi_n$, where $\xi_n\in B_1$, be such that $u(y_n)=0$. For each $s>0$ consider the function $\phi_s\in C^{\infty}_c(B(y_n,2s))$ such that $0\le\phi_s\le1$, $\phi_s=1$ on $B(y_n,s)$ and $\|\nabla\phi_s\|_{L^\infty}\le\frac{C_d}{s}$, where $C_d$ is some constant depending only on the dimension $d$. Thus, we have that 
\begin{equation}
|\langle\Delta u+f,\phi_s\rangle|\le C_dC_bs^{d-1},
\end{equation} 
where $C$ is the constant from \eqref{b}. Denote with $\mu_1$ and $\mu_2$ the positive Borel measures $\Delta u^++fI_{\{u>0\}}$ and $\Delta u^--fI_{\{u<0\}}$. Then, we have
\begin{equation}
\mu_1(B_s(y_n))\le\langle\mu_1,\phi_s\rangle=\langle\mu_1-\mu_2,\phi_s\rangle+\langle\mu_2,\phi_s\rangle\le C_dC_bs^{d-1}+\mu_2(B_{2s}(y_n)).
\end{equation}
Moreover, since $f\in L^{\infty}$, we have that for each $s\le1$,
\begin{equation}
\Delta u^+(B_s(y_n))\le \big(C_dC_b+(1+2^d)\|f\|_\infty\big)s^{d-1}+\Delta u^-(B_{2s}(y_n)).
\end{equation}
Multiplying by $s^{1-d}$ and integrating, we obtain
\begin{equation}
\mean{\partial B_{\delta_n}(y_n)}{u^+\,d\H^{d-1}}\le\frac12\mean{\partial B_{2\delta_n}(y_n)}{u^-\,d\H^{d-1}}+\big(C_dC_b+(1+2^d)\|f\|_\infty\big)\delta_n,
\end{equation}
or, equivalently,
\begin{equation}
\mean{\partial B_1}{u_n^+(\xi_n+\cdot)\,d\H^{d-1}}\le\mean{\partial B_2}{u^-(\xi_n+\cdot)\,d\H^{d-1}}+\big(C_dC_b+(1+2^d)\|f\|_\infty\big)\delta_n.
\end{equation}
Since, the right-hand side goes to zero as $n\to\infty$, so does the left-hand side. Up to a subsequence, we may assume that $\xi_n\to\xi_\infty$ and so, $u_n(\xi_n+\cdot)\to u_\infty(\xi_\infty+\cdot)=u_\infty$ in $H^1_{loc}(\R^d)$. Thus $u_\infty=0$.\\
\emph{Step 5. The convergence $u_n\to 0$ is uniform on the ball $B_1$.}\\
We already know that $u_n\to 0$ in $H^1_{loc}(\R^d)$. Moreover, by the same argument as in Step 3, we have that
\begin{equation}
-\Delta |u_n|\le \delta_n^2|f|,
\end{equation}
 in $\R^d$ and that $|u_n|\to 0$ uniformly on any ball.
\end{proof} 

\begin{oss}\label{rmkcont}
In $\R^2$ the continuity of the state function $u$, from Theorem \ref{main}, can be deduced by the classical Alt-Caffarelli argument, which we apply after reducing the problem to the case when $u$ is positive. 
For example, if $u\in H^1(\R^2)$ is a function satisfying  
$$J_{\lambda u}(u)+c|\{u\neq0\}|\le J_{\lambda u}(v)+c|\{v\neq 0\}|,\qquad\forall v\in H^1(\R^2),$$
for some $\lambda>0$, then $u$ is continuous. Indeed, let $x_0\in\R^d$ be such that $u(x_0)>0$ and let $r_0>0$ and $\eps>0$ be small enough such that, for every $x\in\R^d$ and every $r\le r_0$, we have $\int_{B_r(x)}|\nabla u|^2\,dx\le \eps$. As a consequence, for every $x\in\R^d$ there is some $r_x\in[r_0/2,r_0]$ such that
$\int_{\partial B_{r_x}(x)}|\nabla u|^2\,dx\le 2\eps/r_0$ and 
\begin{equation}\label{osc}
\hbox{osc}_{\partial B_{r_x}(x)}u\le \int_{\partial B_{r_x}(x)}|\nabla u|\,d\H^1\le \sqrt{2\pi r_0}\sqrt{2\eps/r_0}\le\sqrt{4\pi\eps}.
\end{equation}
On the other hand, the positive part $u^+=\sup\{u,0\}$ of $u$ satisfies $\Delta u^++\lambda\|u\|_\infty\ge 0$ on $\R^d$ and so, there is a constant $C>0$ such that
$$u(x_0)\le \mean{\partial B_{r_{x_0}}(x_0)}{ u\,d\H^1}+ Cr_{x_0}^2,$$
which together with \eqref{osc} gives that, choosing $r_0>0$ small enough, we can construct a ball $B_r(x_0)$ of radius $r\le r_0$ such that $u\ge u(x_0)/2>0$ on $\partial B_r(x_0)$. 

We then notice that the set $\{u<0\}\cap B_r(x_0)$ has measure $0$. Indeed, if this is not the case, then the function $\widetilde u=\sup\{-u,0\}\ind_{B_r(x_0)}\in H^1_0(B_r(x_0))$ is such that $J_{\lambda u}(u)=J_{\lambda u}(-\widetilde u)+J_{\lambda u}(u+\widetilde u)$. By the maximum principle $\|\widetilde u\|_\infty\le Cr_0^2$ and so, for some constant $C>0$, we have 
$$ \big|J_{\lambda u}(-\widetilde u)\big|\le Cr_0^2\big|\{u<0\}\cap B_r(x_0)\big|< c\big|\{u<0\}\cap B_r(x_0)\big|,$$
for $r_0$ small enough.
Hence we have $J_{\lambda u}(-\widetilde u)+c|\{u<0\}\cap B_r(x_0)|>0$, that contradicts the quasi-minimality of $u$.

We conclude the proof by showing that the set $\{u=0\}\cap B_{r}(x_0)$ has measure $0$. We compare $u$ with the function $w=\ind_{B^c_{r}(x_0)}u+\ind_{B_{r}(x_0)}v$, where $v$ is the function from \eqref{altcafv}. 
\begin{align*}
c\big|\{u=0\}\cap B_r(x_0)\big|&\ge J_{\lambda u}(u)-J_{\lambda u}(w)\\
&=\frac12\int_{B_{r}(x_0)}\big(|\nabla u|^2-|\nabla v|^2\big)\,dx-\int_{B_{r}(x_0)}\lambda u(u-v)\,dx\\
&\ge \frac12 \int_{B_{r}(x_0)}|\nabla (u-v)|^2\,dx\\
&\ge \frac{C_2}{r^2}\big|\{u=0\}\cap B_r(x_0)\big|\Big(\mean{\partial B_{r}(x_0)}{u\,d\H^{1}}\Big)^2,
\end{align*} 
where the last inequality is due to \eqref{altcafineq}. If we suppose that $|\{u=0\}\cap B_r(x_0)|>0$, then for some constant $C>0$, we would have $u(x_0)\le  Cr_0^2$, which is absurd choosing $r_0>0$ small enough.

\end{oss}

\begin{lemma}\label{l2}
Let $u\in H^1(\R^d)$ satisfies the conditions $(a)$ and $(b)$ from Theorem \ref{main}. Then, for each  $x_0\in\R^d$, in which $u$ vanishes, and each $0<r\le r_0/4$, where $r_0$ is the constant from condition $(b)$ in Theorem \ref{main}, we have that
\begin{equation}
|\Delta |u||(B_r(x_0))\le C\,r^{d-1},
\end{equation}
where the constant $C$ is given by the expression $C=C_d(C_b+\sqrt{C_m}+1)$, where $C_d$ is a constant depending only on the dimension, $C_b$ is the constant from \eqref{b} and $C_m$ is the constant from the monotonicity formula \eqref{mon}.   
\end{lemma}
\begin{proof}
Without loss of generality we can suppose $x_0=0$. For each $r>0$, consider the functions 
$$v^r:=v^r_+-v^r_-,\qquad w^r:=w^r_+-w^r_-,$$
where $v^r_\pm$ and $w^r_\pm$ are solutions of the following equations on $B_r$
\begin{equation}\label{l2e1}
\begin{array}{lll}
\begin{cases}
\begin{array}{rr}
-\Delta v^r_\pm=f^\pm,\ \hbox{in}\ B_r,\\
v^r_\pm=u^\pm,\ \hbox{on}\ \partial B_r,
\end{array}
\end{cases} & 
\begin{cases}
\begin{array}{rr}
-\Delta w^r_\pm=f^\pm,\ \hbox{in}\ B_r,\\
v^r_\pm=0,\ \hbox{on}\ \partial B_r.
\end{array}
\end{cases} 
\end{array}
\end{equation}

Thus we have that $v^r_\pm-w^r_\pm$ is harmonic in $B_r$ and so, the estimate
\begin{equation}\label{l2e2}
\int_{B_r}|\nabla(v^r_\pm-w^r_\pm)|^2\,dx\le \int_{B_r}|\nabla u^\pm|^2\,dx.
\end{equation}
Since $u^\pm-v^r_\pm+w^r_\pm\in H^1_0(B_r)$, we have 
\begin{align}
\int_{B_r}|\nabla (u^\pm-v^r_\pm+w^r_\pm)|^2\,dx & =\int_{B_r}\nabla u^\pm\cdot\nabla (u^\pm-v^r_\pm+w^r_\pm)\,dx\nonumber\\
&=\int_{B_r}|\nabla u^\pm|^2\,dx+\int_{B_r}\nabla u^\pm\cdot\nabla (w^r_\pm-v^r_\pm)\,dx\\
&\le 2\int_{B_r}|\nabla u^\pm|^2\,dx,\nonumber 
\end{align}
where the last inequality is due to \eqref{l2e2}. Thus, we obtain
\begin{equation}\label{l2e3}
\begin{array}{lll}
\ds\left(\mean{B_r}{|\nabla (u^+-v^r_++w^r_+)|^2\,dx}\right)\left(\mean{B_r}{|\nabla (u^--v^r_-+w^r_-)|^2\,dx}\right)\\
\ds\qquad\qquad\qquad\qquad\qquad\le 4\left(\mean{B_r}{|\nabla u^+|^2\,dx}\right)\left(\mean{B_r}{|\nabla u^-|^2\,dx}\right)\\
\ds\qquad\qquad\qquad\qquad\qquad\le 4C_m,
\end{array}
\end{equation}
where the last inequality is due to the monotonicity formula \eqref{mon} and $C_m$ is the constant that appears there.\\
On the other hand, for $0<r\le r_0\le1$, we have 
\begin{equation}\label{l2e4}
\begin{array}{lll}
\ds\int_{B_r}|\nabla (u-v^r+w^r)|^2\,dx&\ds \le 2\int_{B_r}|\nabla (u-v^r)|^2\,dx+2\int_{B_r}|\nabla w^r|^2\,dx\\
&\ds=2\int_{B_r}\left[\nabla u\cdot\nabla(u-v^r)+f(u-v_r)\right]\,dx+2\int_{B_r}|\nabla w^r|^2\,dx\\
&\ds\le C_b^2r^d+C_dr^d,
\end{array}
\end{equation}
where $C_b$ is the constant from condition $(b)$. Using \eqref{l2e3} and \eqref{l2e4}, we have

\begin{equation}\label{l2e5}
\begin{array}{lll}
\ds\int_{B_r}|\nabla (u^+-v^r_++w^r_+)|^2\,dx+\int_{B_r}|\nabla (u^--v^r_-+w^r_-)|^2\,dx\\
\ds\qquad\le \int_{B_r}|\nabla (u-v^r+w^r)|^2\,dx\\
\ds\qquad\qquad+2\left(\int_{B_r}|\nabla (u^+-v^r_++w^r_+)|^2\,dx\right)^{1/2}\left(\int_{B_r}|\nabla (u^--v^r_-+w^r_-)|^2\,dx\right)^{1/2}\\
\ds\qquad\le(C_b^2+4C_m+C_d)r^d.
\end{array}
\end{equation}

Denoting with $C_{b,m,d}$ the constant
\begin{equation}
C_{b,m,d}=2C_b^2+8C_m=C_d,
\end{equation}
we have the estimate
\begin{equation}\label{l2e6}
\int_{B_r}|\nabla (u^\pm-v^r_\pm)|^2\,dx\le C_{b,m,d}r^d.
\end{equation}

Note that $u^+\le v^r_+$. In fact, we have
\begin{equation}\label{l2e7}
\Delta(u^+-v^r_+)=\Delta u^++f^+\ge\Delta u^++fI_{ \{u>0\} },
\end{equation} 
and so, $u^+-v^r_+$ is sub-harmonic in $B_r$ and vanishes on $\partial B_r$ and thus, is negative. Analogously, $\Delta (u^--v^r_-)\ge \Delta u^--fI_{\{u<0\}}$ and $u^-\le v^r_-$. Moreover, by \eqref{l2e7} and the fact that $u^+-v^r_+\in H^1_0(B_r)$, we have that 
\begin{equation}\label{l2e8}
\begin{array}{lll}
\ds\int_{B_r}|\nabla (u^+-v^r_+)|^2\,dx&\ds\ge\int_{B_r}-\nabla(v^r_+-u^+)\cdot\nabla u^++(v^r_+-u^+)fI_{\{u>0\}}\,dx\\
\\
&\ds=\int_{B_r}(v^r_+-u^+)\,d\mu_1=\int_{B_r}v^r_+\,d\mu_1.
\end{array}
\end{equation} 
Applying the estimate \eqref{l2e6} and setting
\begin{equation}\label{l2e9}
\mu_1:=\Delta u^++fI_{\{u>0\}},\qquad\mu_2:=\Delta u^--fI_{\{u<0\}}, 
\end{equation}
we have that 
\begin{equation}\label{l2e10}
\int_{B_r}v^r_+\,d\mu_1\le C_{b,m,d}r^d,\qquad\qquad\int_{B_r}v^r_-\,d\mu_2\le C_{b,m,d}r^d
\end{equation}

Setting $U:=u^+-v^r_+\le 0$ on $B^r$, we have that for each $z\in B_{r/4}$
\begin{equation}\label{l2e11}
\mean{\partial B_{3r/4}(z)}{U\,d\H^{d-1}}\le 0\le u^+(z)=U(z)+v^r_+(z).
\end{equation}

Applying \eqref{pwde4} to $U\in H^1(B_r)$ and using \eqref{l2e7}, we obtain
\begin{equation}\label{l2e12}
\begin{array}{lll}
\ds v^r_+(z)&\ds \ge \int_0^{3r/4}s^{1-d}\int_{B_s(z)}\Delta U(B_s(z))\,ds\\
&\ds \ge \int_0^{3r/4}s^{1-d}\int_{B_s(z)}\mu_1(B_s(z))\,ds.
\end{array}
\end{equation}
Integrating both sides of \eqref{l2e12} on $B_{r/4}$ with respect to $d\mu_1(z)$, we obtain
\begin{equation}\label{l2e13}
\begin{array}{lll}
\ds C_{b,m,d}(r/4)^d &\ds\ge\int_{B_{r/4}}v^r_+(z)\,d\mu_1(z)\\
\\
&\ds\ge\frac{1}{d\omega_d}\int_{B_{r/4}}d\mu_1(z)\,\int_0^{3r/4}s^{1-d}\mu_1(B_s(z))\,ds\\
\\
&\ds\ge\frac{1}{d\omega_d}\int_{B_{r/4}}d\mu_1(z)\,\int_{r/2}^{3r/4}s^{1-d}\mu_1(B_s(z))\,ds\\
\\
&\ds\ge\frac{1}{d\omega_d}\int_{B_{r/4}}d\mu_1(z)\,\int_{r/2}^{3r/4}s^{1-d}\mu_1(B_{r/4})\,ds\\
\\
&\ds\ge C_dr^{2-d}\left[\mu_1(B_{r/4})\right]^2,  
\end{array} 
\end{equation}
which proves the claim.
\end{proof}

\begin{proof}[Proof of Theorem \ref{main}]
Note that we can assume $\Omega=\{u\neq0\}$. Since, by Lemma \eqref{l1}, $u:\R^d\to\R$ is continuous, we have that $\Omega:=\{u\neq0\}$ is open. For any $r>0$, denote with  $\Omega_r\subset\Omega$ the set $\left\{x\in\omega:\,d(x,\Omega^c)<r\right\}$. Choose $x\in\omega_{r_0/2}$ and let $y\in\partial\Omega$ such that $R_x:=|x-y|=d(x,\Omega^c)$. We use the gradient estimate from Remark \ref{gradest} of $u$ on the ball $B_{R_x}(x)$:
\begin{equation}\label{maine1}
\begin{array}{lll}
\ds|\nabla u(x)|&\ds\le C_d\|f\|_{L^\infty}+\frac{2d}{R_x}\|u\|_{L^\infty(B_{R_x}(x))}\\
\\
&\ds\le C_d\|f\|_{L^\infty}+\frac{2d}{R_x}\|u\|_{L^\infty(B_{2R_x}(y))}\\
\\
&\ds\le (C_d+r_0)\|f\|_{L^\infty}+\frac{C_d}{R_x}\mean{\partial B_{3R_x}(y)}{|u|\,d\H^{d-1}}\\
\\
&\ds\le (C_d+r_0)\|f\|_{L^\infty}+\frac{C_d}{R_x}\int_0^{3R_x}s^{1-d}|\Delta |u||(B_s(y))\,ds\\
\\
&\ds\le (C_d+r_0)\|f\|_{L^\infty}+3C_d C,
\end{array}
\end{equation} 
where $C=C_d(C_b+\sqrt{C_m}+1)$ is the constant from Lemma \ref{l2}. Since for $x\in\Omega\setminus\Omega_{r_0/2}$, we have that 
\begin{equation}\label{maine2}
|\nabla u(x)|\le C_d\|f\|_{L^\infty}+\frac{4d}{r_0}\|u\|_{L^\infty},
\end{equation} 
we obtain that $u$ is Lipschitz and
\begin{equation}\label{maine3}
\|\nabla u\|_{L^{\infty}}\le (C_d+r_0)\|f\|_{\infty}+C_d\max\left\{C_b+\sqrt{C_m}+1,\,\frac{\|u\|_\infty}{r_0}\right\}.
\end{equation}
\end{proof}


\bigskip
{\small\noindent
Dorin Bucur:
Laboratoire de Math\'ematiques (LAMA),
Universit\'e de Savoie\\
Campus Scientifique,
73376 Le-Bourget-Du-Lac - FRANCE\\
{\tt dorin.bucur@univ-savoie.fr}\\
{\tt http://www.lama.univ-savoie.fr/$\sim$bucur/}

\bigskip\noindent
Dario Mazzoleni:
Dipartimento di Matematica, Universit\`a degli Studi di Pavia\\
Via Ferrata, 1, 27100 Pavia - ITALY \\
{\texttt{dario.mazzoleni@unipv.it}} \\
Department Mathematik, Friederich-Alexander Universit\"at Erlangen-N\"urnberg \\
Cauerstrasse,11, 91058 Erlangen - GERMANY\\
{\texttt{mazzoleni@math.fau.de}}

\bigskip\noindent
Aldo Pratelli:
Department Mathematik, Friederich-Alexander Universit\"at Erlangen-N\"urnberg \\
Cauerstrasse,11, 91058 Erlangen - GERMANY\\
{\texttt{pratelli@math.fau.de}}

\bigskip\noindent
Bozhidar Velichkov:
Scuola Normale Superiore di Pisa\\
Piazza dei Cavalieri 7, 56126 Pisa - ITALY\\
{\tt b.velichkov@sns.it}

\end{document}